\documentclass{article}

\usepackage{amsfonts}   
\usepackage{amsmath}    
\usepackage{amsthm}     
\usepackage{dsfont}     
\usepackage{mathrsfs}   

\usepackage[usenames,dvipsnames]{color}

\newcommand{\E}{\mathbb{E}}

\newtheorem{define}{Definition}[section]
\newtheorem{lemma}[define]{Lemma}
\newtheorem{claim}[define]{Claim}
\newtheorem{cor}[define]{Corollary}
\newtheorem{thm}[define]{Theorem}
\newtheorem{prop}[define]{Proposition}

\usepackage{hyperref}
\hypersetup{
  linkcolor=blue,        
  citecolor=blue,        
  colorlinks=true,       
  linktocpage=true,
}

\newcommand{\Var}{\mathrm{Var}}
\newcommand{\induce}{\subseteq_i}

\begin{document}

\title{Random Perfect Graphs}

\author{
  Colin McDiarmid\\
  \texttt{cmcd@stats.ox.ac.uk}\\
    Department of Statistics\\
    University of Oxford
  \and
  Nikola Yolov\\
  \texttt{nikola.yolov@cs.ox.ac.uk}\\
  Department of Computer Science\\
  University of Oxford
}
\maketitle

\begin{abstract}
  We investigate the asymptotic structure of a random perfect graph $P_n$
  sampled uniformly from the set of perfect graphs on vertex set $\{1,\ldots,n\}$.
  Our approach is based on the result of Pr{\"o}mel and Steger
  that almost all perfect graphs are generalised split graphs,
  together with a method to generate such graphs almost uniformly.

  We show that the distribution of the maximum of the stability number
  $\alpha(P_n)$ and clique number $\omega(P_n)$ is close to
  a concentrated distribution $L(n)$
  which plays an important role in our generation method.
  We also prove that the probability that $P_n$ contains any given graph $H$
  as an induced subgraph is asymptotically $0$ or $\frac12$ or $1$.
  Further we show that almost all perfect graphs
  are $2$-clique-colourable, improving a result of Bacs\'o et al from 2004;
  they are almost all Hamiltonian;
  they almost all have connectivity $\kappa(P_n)$ equal to their minimum degree;
  they are almost all in class one
  (edge-colourable using $\Delta$ colours, where $\Delta$ is the maximum degree);
  and a sequence of independently and uniformly sampled perfect graphs
  of increasing size converges almost surely to the graphon
  $W_P(x, y) = \frac12(\mathds{1}[x \le 1/2] + \mathds{1}[y \le 1/2])$.
\end{abstract}

\section{Introduction}

A graph is \emph{perfect} if the chromatic number equals the clique number
in each of its induced subgraphs.
Perfect graphs have formed a central field in graph theory for some decades,
partly because of the challenging open problems
and partly because of the connections to polyhedral combinatorics,
linear optimisation and computational complexity,
see for example \cite{reed_perfect_book}.
Fundamental classes of graphs that are all perfect include
bipartite graphs, chordal graphs, comparability graphs and interval graphs.

Generalised split graphs form another important class of perfect graphs.
A graph $G$ is \emph{unipolar} if for some $k \geq 0$ its vertex set $V(G)$
can be partitioned into $k + 1$ cliques
$C_0, C_1, \ldots C_k$, so that there are no edges between $C_i$ and $C_j$
for $1 \le i < j \le k$.
We call $C_0$ the \emph{central clique},
and the $C_i$ for $i \geq 1$ the \emph{side cliques};
and we call the pair $(G, C_0)$ a \emph{unipolar arrangement of order} $v(G)$.
A graph $G$ is \emph{co-unipolar} if its complement $\overline{G}$ is unipolar; and it
is a \emph{generalised split graph} if it is unipolar or co-unipolar.
We denote the classes of perfect, unipolar, co-unipolar and generalised split graphs
by $\mathcal{P}$, $\mathcal{GS}^+$, $\mathcal{GS}^-$ and $\mathcal{GS}$ respectively.
Given a class $\mathcal G$ of graphs,
we let ${\mathcal G}_n$ denote the set of graphs in $\mathcal G$
on vertex set $[n]:=\{1,\ldots,n\}$.
We say that a sequence of events $(A_n)_{n \ge 1}$
holds \emph{with high probabilty (or whp)} if $\mathbb{P}(A_n) = 1 - o(1)$,
and that it holds \emph{with very high probabilty (or wvhp)}
if $\mathbb{P}(A_n) = 1 - e^{-\Omega(n)}$.

Our goal is to describe the asymptotic properties of perfect graphs.
We rely on the key theorem of Pr{\"o}mel and Steger~\cite{promelsteger}
that almost all perfect graphs are generalised split graphs.
We present and analyse an almost uniform generation process for the graphs in
$\mathcal{GS}_n$.
This generation process, together with the Pr{\"o}mel-Steger theorem,
yields a powerful method for working with uniformly sampled perfect graphs.
The usefulness of the method will be seen when we prove a range of results
about the asymptotic behaviour of perfect graphs, described in the next subsection.

\subsection{Plan of the paper}
\label{subsec.plan}
In \S~\ref{sec:generation}
we present our method for analysing random perfect graphs.
The main result, Theorem~\ref{thm:main}, states that the total variation distance between
a naturally generated random generalised split graph in $\mathcal{GS}_n$
and the uniformly sampled perfect graph $P_n$ is $e^{-\Theta(n)}$.
We introduce a 
family of concentrated distributions $L(n)$ which are used in the generation process, and give a number of results essential for our analysis, including surveying
relevant results about random partitions, but we defer the
detailed proofs concerning the generation process to \S~\ref{sec:genlemma}
near the end of the paper.

In \S~\ref{sec:stabcl}
we study the stability number $\alpha(P_n)$  and clique number $\omega(P_n)$.
We show that the minimum of the two numbers is asymptotically normally distributed
(with specified mean and variance),
and the distribution of the maximum is very close to the distribution $L(n)$ mentioned above;
see Theorem~\ref{thm:homogeneous}.

In \S~\ref{sec:subgraphs}
we discuss which induced subgraphs are likely to be found in a random perfect graph,
and we find a trichotomy.
Theorem~\ref{thm:subgraph_lemma} states that
depending on whether
(i) $H \in \mathcal{P} \setminus \mathcal{GS}$,
(ii) $H \in \mathcal{GS} \setminus \mathcal{GS}^-$ or
$H \in \mathcal{GS} \setminus \mathcal{GS}^+$ and
(iii) $H \in \mathcal{GS}^+ \cap \mathcal{GS}^-$,
the probability that $P_n$ contains an induced copy of $H$ is respectively
(i) $e^{-\Theta(n)}$,
(ii) $1/2 \pm e^{-\Omega(n)}$ and
(iii) $1 - e^{-\Omega(n)}$.
For example, the complete bipartite graph
$K_{2,3} \in \mathcal{GS}^- \setminus \mathcal{GS}^+$,
so the probability that $P_n$ contains an induced $K_{2,3}$ is well-estimated by $\frac12$.

In \S~\ref{sec:clique} we focus on clique colourings, that is colourings of the vertices such that no maximal clique is monochromatic (ignoring isolated vertices).
The main result is Theorem~\ref{thm:cliques},
stating that almost all perfect graphs are 2-clique colourable,
improving the result of Bacs{\'o} et al \cite{cliques}
from 2004 that almost all perfect graphs are 3-clique colourable.

In \S~\ref{sec:hamiltonian}
we discuss Hamilton cycles, and find that the distribution
$L(n)$ appears again.  We show that, if  $X \sim L(n)$, the probability that $P_n$
is Hamiltonian is
\begin{equation} \label{eqn.haminsec1}
1 - \frac{1}{2}\mathbb{P}(X > n/2) \pm e^{-\Omega(n)}
= 1-2^{-\frac14 (1 + o(1)) \log^2 n} = 1 - o(1).
\end{equation}
The precise statement is given in Theorem~\ref{thm:hamilt}.

In \S~\ref{sec:conn-chrindex}
we discuss connectivity and the chromatic index of perfect graphs.
We show that for a random perfect graph $P_n$
the connectivity $\kappa(P_n)$ equals the minimum degree wvhp,
Theorem~\ref{thm:connectivity};
and the chromatic index equals the maximum degree whp,
Theorem~\ref{thm:chromatic_idx}.

In \S~\ref{sec:limits} we study the limit of a sequence of independently
sampled perfect graphs $P_n$.
A \emph{graphon} is a symmetric measurable function $W : [0,1]^2 \to [0, 1]$,
that could be seen as the adjacency matrix of
a graph with the set $[0, 1]$ as vertices.
The \emph{cut distance}, $\delta_{\square}$,
is a metric over the space of graphons, $\mathcal{W}_0$,
generalising the maximum absolute value of a rectangular difference
of two matrices to graphons.
The space $(\mathcal{W}_0, \delta_\square)$ has many important properties.  In particular, it is compact, under the assumption that graphons with cut distance $0$ are identified. This result has interesting corollaries, including some of the strongest generalisations of the Szemer\'{e}di regularity lemma.
\noindent
Theorem~\ref{thm:limit_main} states that
\begin{align*}
  &\mathbb{P}(\delta_{\square}(P_n, W_P) \le n^{-1/2})
  = 1 - e^{-\Omega(\sqrt{n} \log n)} \text{ and}\\
  &\mathbb{P}(\delta_{\square}(P_n, W_P) \le (\log n)^{-2})
  = 1 - e^{-\Theta(n)},
\end{align*}
where $W_P(x, y)$ is the graphon
$ \frac12(\mathds{1}[x \le 1/2] + \mathds{1}[y \le 1/2])$.
This implies that a sequence of uniformly and independently sampled
perfect graphs of increasing size converges
almost surely to $W_P$
and provides an analytic tool for estimating subgraph densities of $P_n$.

In \S~\ref{sec:genlemma} we give the proofs concerning the generation process which were
deferred from \S~\ref{sec:generation}; and finally in \S~\ref{sec:concl}
we make a few concluding remarks and mention some open problems.

\section{Generating perfect graphs}
\label{sec:generation}
Before we go any further, let us check that generalised split graphs are perfect.
The weak perfect graph theorem states
that the complement of a perfect graph is perfect \cite{weak_perfect1, weak_perfect2},
so it suffices to consider a unipolar graph $G$, and show it is perfect.
Also, each induced subgraph of a unipolar graph is unipolar,
so it suffices to show that $\chi(G)=\omega(G)$.

Consider a unipolar arrangement $(G, C_0)$, where $C_0$ is the central clique,
and let $C_1, \ldots, C_k$ be the side cliques.
Let $G_j$ denote the induced subgraph $G[C_0 \cup C_j]$.
Bipartite graphs are perfect, so co-bipartite graphs are perfect,
and thus $\chi(G_j)=\omega(G_j)$ for each $j$.
Therefore we can colour each $G_j$ with $\max_j \omega(G_j)$ colours,
and then make sure that the different colourings agree on the central clique $C_0$.
It is easy to see that $\omega(G)= \max_j \omega(G_j)$
and it follows that $\chi(G)=\omega(G)$, as required.
\smallskip

We next discuss some probability distributions on the set $\mathcal{G}_n$
of all graphs on $[n]$.
Probably the most well-known
is that of the binomial random graph $G(n, p)$,
which assigns probability $p^{e(G)}(1-p)^{e(\overline{G})}$
to each graph $G$ in ${\mathcal G}_n$.
We are predominantly interested in the following
probability distribution on ${\mathcal G}_n$:
\[
\mathbb{P}_{\mathcal{P}_n}(G)
= \begin{cases} 1 / |\mathcal{P}_n| & \mbox{if } G \in \mathcal{P}_n \\
  0                   & \mbox{otherwise.} \end{cases}
\]
Here is another distribution on ${\mathcal G}_n$:
\[
\mathbb{P}_{\mathcal{GS}_n}(G)
= \begin{cases} 1 / |\mathcal{GS}_n| & \mbox{if } G \in \mathcal{GS}_n \\
  0                    & \mbox{otherwise.} \end{cases}
\]
We rely heavily on Theorem~2.4 from \cite{promelsteger},
which for our purposes can be reformulated as follows:
\begin{equation}
  \label{eq:pre_promel_steger}
  |\mathcal{GS}_n| = (1 - e^{-\Omega(n)})|\mathcal{P}_n|.
\end{equation}
In other words, the class of $\mathcal{GS}$-graphs
is a very good asymptotic approximation to the class of perfect graphs,
as all $\mathcal{GS}$-graphs are perfect
and almost all perfect graphs are $\mathcal{GS}$-graphs.
The bound in (\ref{eq:pre_promel_steger}) is sharp.
Fix any graph $U \in \mathcal{P} \setminus \mathcal{GS}$,
for example take the disjoint union of $K_{2, 3}$ and $K_3$,
and consider the family $\mathcal{F}^U$ of the graphs that are
a disjoint union of $U$ and a perfect graph.
We have $\mathcal{F}^U \subseteq \mathcal{P} \setminus \mathcal{GS}$
and $|\mathcal{F}^U_n| \ge |\mathcal{P}_{n - v(U)}|$,
which combined with estimates on the size of $\mathcal{P}_n$,
e.g. $|\mathcal{P}_n|/|\mathcal{P}_{n-1}| = 2 ^{(1+o(1))n/2}$,
implies
\begin{equation}
  \label{eq:promel_steger}
  |\mathcal{GS}_n| = (1 - e^{-\Theta(n)})|\mathcal{P}_n|.
\end{equation}

We can also express (\ref{eq:promel_steger}) using total variation distance:
\begin{align}
  \label{eq:dtv_perfect_gs}
  d_{TV}(\mathbb{P}_{\mathcal{GS}_n}, \mathbb{P}_{\mathcal{P}_n})
  &= \sup_{ A \subseteq {\cal G}_n}
  | \mathbb{P}_{\mathcal{GS}_n}(A) - \mathbb{P}_{\mathcal{P}_n}(A) | \nonumber
  = \mathbb{P}_{\mathcal{P}_n}(\mathcal{P}_n \setminus \mathcal{GS}_n) \\
  &= \frac{|\mathcal{P}_n \setminus \mathcal{GS}_n|}{|\mathcal{P}_n|}
  = e^{-\Theta(n)}.
\end{align}

How many $n$-vertex perfect graphs are there?
The number is very close to the number of $n$-vertex $\mathcal{GS}$-graphs,
which is approximately twice the number of $n$-vertex unipolar graphs.
The last class is easy to estimate: for each $k$ between $0$ and $n$
there are $n \choose k$ ways to choose which $k$ vertices to include in the central clique,
$B_{n-k}$ ways to specify the partition of the remaining vertices and
$2^{k(n-k)}$ ways to specify the edges between the central clique and the remaining vertices.
Here $B_n$ stands for the $n$-th Bell number,
which is the number of ways to partition an $n$-element set.
Note that the last calculation is not precise,
since unipolar graphs may be part of more than one unipolar arrangement.
If we denote ${n \choose k}2^{k(n-k)}B_{n-k}$ by $\ell_{n, k}$,
we would expect that the number of unipolar graphs is close to
$\mathscr{L}_n := \sum_{k = 0}^n \ell_{n, k}$.
This turns out to be a very precise estimate and we show in \S~\ref{sec:genlemma} that
\begin{align}
  \label{eq:double_counting}
  |\mathcal{GS}_n| = 2 \mathscr{L}_n (1 - e^{-\Omega(n)}).
\end{align}
(See Lemma 2.2 of~\cite{promelsteger} for an asymptotic estimate of
$|\mathcal{GS}_n|$.)

\begin{define}
  \label{def:ln}
  For each $n$ let $L(n)$ be the discrete integer-valued distribution
  with probability mass function
  \[
  p_{L(n)}(x) = \ell_{n, x} / \mathscr{L}_n
  \hspace{20pt} \text{ for } x= 0,\ldots, n.
  \]
\end{define}

\subsection{The generation process}
\label{subsec.genproc}

Now we describe our method to sample almost uniformly
from the set of generalised split graphs on $[n]$.

\begin{define}
  For each $n$ let $Gen(n)$ be a random quadruple $(B, E, (k, \sigma), \pi)$, such that
  \begin{enumerate}
  \item
    $B$ is a $\{-1, 1\}$-valued random variable, taking each value with probability half,
  \item
    $E$ is the set of edges of an Erd\H{o}s-R\'{e}nyi random graph $G(n, \frac12)$,
  \item
    $(k, \sigma)$ is the result of the following random experiment:
    choose a number $k \in [0, n]$ with distribution $L(n)$,
    set $W=\{k+1,\ldots,n\}$,
    and then sample $\sigma$ uniformly at random from $\Pi(W)$,
    the set of all partitions of $W$,
  \item
    $\pi$ is a uniformly sampled permutation from $Sym([n])$,
  \item
    the components of the quadruple are independent.
  \end{enumerate}
\end{define}

Let $\rho$ be the (deterministic) map from such quadruples
to $\mathcal{G}_n$ defined as follows.
Let $U = (B, E, (k, \sigma), \pi)$ be an outcome of $Gen(n)$.
Start with an $n$-vertex empty graph $G_0 = ([n], \emptyset)$.
It is convenient to write $C = [k]$ and $S = V(G) \setminus C$.
Create $G_1$ by adding edges to $G_0$,
so that $C$ induces a clique,
and $S$ induces a disjoint union of cliques corresponding to $\sigma$,
i.e. $ij \in E(G_1)$ for $1 \le i < j \le k$;
and $ij \in E(G_1)$ whenever $i$ and $j$ are contained in the same part of
$\sigma$ for $k < i < j \le n$.
Let $G_2$ be obtained from $G_1$ by adding the edges of $E$ between $C$ and $S$,
and let $G_3$ be $G_2$ if $B = 1$ and $\overline{G_2}$ otherwise.
Finally let $G_4$ be obtained by permuting the vertices of $G_3$ according to $\pi$,
i.e. $\pi(i)\pi(j) \in E(G_4)$ iff $ij \in E(G_3)$.
We define $\rho(U) = G_4$ for each such vector $U$.

Consider $U = (B, E, (k, \sigma), \pi) \sim Gen(n)$
and the corresponding random graph $G = \rho(U)$.
It is easy to see that $G$ is unipolar if $B = 1$ and co-unipolar if $B = -1$.
Indeed, $(G, \pi(C))$ is a unipolar arrangement if $B = 1$
and $(\overline G, \pi(C))$ is a unipolar arrangement otherwise.
We often refer to the appropriate pair as the \emph{induced arrangement}
of $\rho(U)$ by the generation.
Whenever we want to test if $G$ has some isomorphism-closed graph property,
$\pi$ does not affect the membership,
so we can assume that $\pi$ is the identity permutation.
All graph properties in which we are interested are isomorphism-closed,
so $\pi$ will largely be ignored.

We define our last probability distribution, $\mathbb{P}_{Gen(n)}$ on ${\mathcal G}_n$,
as follows:
\[
\mathbb{P}_{Gen(n)}(G) = \mathbb{P}(\rho(Gen(n)) =G).
\]
By the above, $\mathbb{P}_{Gen(n)}(G) > 0$ if and only if $G \in \mathcal{GS}_n$.
We shall show that
\begin{equation}
  \label{eq:dtv_gs_gen}
  d_{TV}(\mathbb{P}_{\mathcal{GS}_n}, \mathbb{P}_{Gen(n)}) = e^{-\Omega(n)}.
\end{equation}

It is not hard to prove (\ref{eq:dtv_gs_gen}) from (\ref{eq:double_counting}).
However, both equations have technical proofs,
so they are postponed to \S~\ref{sec:genlemma}.
\begin{thm}
  \label{thm:main}
  The uniform measure $\mathbb{P}_{\mathcal{P}_n}$
  and the generated measure $\mathbb{P}_{Gen(n)}$ satisfy
  \[
  d_{TV}(\mathbb{P}_{\mathcal{P}_n}, \mathbb{P}_{Gen(n)}) = e^{-\Theta(n)}.
  \]
\end{thm}
\begin{proof}
  From the triangle inequality for total variation distance,
  (\ref{eq:dtv_perfect_gs}) and (\ref{eq:dtv_gs_gen}), we have
  \begin{align*}
    d_{TV}(\mathbb{P}_{\mathcal{P}_n}, \mathbb{P}_{Gen(n)})
    &\le d_{TV}(\mathbb{P}_{\mathcal{P}_n}, \mathbb{P}_{\mathcal{GS}_n})
    + d_{TV}(\mathbb{P}_{\mathcal{GS}_n}, \mathbb{P}_{Gen(n)}) =  e^{-\Theta(n)}.
  \end{align*}
  A lower bound follows from arguments similar to these for (\ref{eq:promel_steger}).
\end{proof}

Theorem~\ref{thm:main} describes the precision of our almost-uniform sampling.  Note that any approach relying on (\ref{eq:promel_steger}) cannot hope to achieve better precision.
The proof is rather technical, and is postponed to Section~\ref{sec:genlemma}. 
The properties of the generated graph $\rho((B, E, (k, \sigma), \pi))$
depend a lot on the value of $B$.
It is often the case that we break the analysis into different cases
depending on the value of $B$.
We introduce two new random $4$-tuples - $Gen^+(n)$ and $Gen^-(n)$,
which are the same as $Gen(n)$, but with overwritten values for $B$:

\begin{define}
  For each $n$ define $Gen^+(n)$ to be $(1, E, (k,\sigma), \pi)$
  and $Gen^-(n)$ to be $(-1, E, (k,\sigma), \pi)$,
  where $E$, $(k,\sigma)$ and $\pi$ are the same as in the definition of $Gen(n)$.
\end{define}

The following observation is easily verified:
Let $Q$ be an arbitrary graph property.
If $G \sim \rho(Gen(n))$, $G^+ \sim \rho(Gen^+(n))$ and $G^- \sim \rho(Gen^-(n))$, then
\begin{equation}
  \label{eq:partition_of_gs}
  \mathbb{P}(G \in Q) = \frac{1}{2}\mathbb{P}(G^+ \in Q) + \frac{1}{2}\mathbb{P}(G^- \in Q).
\end{equation}
Still with $G \sim \rho(Gen(n))$, we shall see in Section~\ref{sec:genlemma} that wvhp $G$ is not both unipolar and co-unipolar, that is
\begin{equation} \label{eqn.smallcap}
 \mathbb{P}(G \in \mathcal{GS}^+ \cap \mathcal{GS}^-) = e^{-\Omega(n)}.
\end{equation}

\subsection{On the parts of the generation process}
\label{subsec.genprocparts}

Theorem~\ref{thm:main} will prove to be very useful.
Let $P_n \in_u \mathcal{P}_n$.
Whenever we want to estimate
$q_n = \mathbb{P}_{\mathcal{P}_n}(Q) = \mathbb{P}(P_n \in Q)$
for some graph property $Q$, we can instead estimate
$q'_n = \mathbb{P}_{Gen(n)}(Q) = \mathbb{P}(\rho(Gen(n)) \in Q)$,
and use the fact that $|q_n - q'_n| = e^{-\Omega(n)}$.
The approximate sampling, $Gen(n)$,
is a convenient method for estimating probability,
because it decomposes the resulting graph into three independent parts --
a central clique, a partition of the remaining vertices
and a bipartite part in between.

\subsubsection{Central clique}
An essential part of the generation process involves determining
the size of the central clique.
For this purpose we introduced the distribution $L(n)$.
Recall that if $X \sim L(n)$ and $x$ is an integer between $0$ and $n$,
then $\mathbb{P}(X = x) = \ell_{n, x} / \mathscr{L}_n$,
where $\ell_{n, x} = {n \choose x}2^{x(n-x)}B_{n-x}$
and $\mathscr{L}_n = \sum_{k = 0}^n \ell_{n, k}$.
The term $2^{x(n-x)}$ dwarfs the other two terms and heavily concentrates the distribution. The term $B_{n-x}$ breaks the symmetry around $n/2$ and tilts the mean towards zero by about $\frac{1}{2}\log n$,
where throughout the paper $\log$ means $\log_2$
and $\ln$ denotes the natural logarithm, $\log_e$.

\begin{thm}
  \label{thm:concentration}
  Suppose $X \sim L(n)$. For sufficiently large $n$ we have
  \[
  2^{-(x+1)^2 -1} \le \mathbb{P}(|X - \mu| \ge x)
  \le 2^{-(x-2)^2 +2} +n^{-n},
  \]
  for each $x > 1$, where
  \[
  \mu = \mu(n) = \frac{n - \log n + \log \ln n}{2}.
  \]
\end{thm}

The upper bound in the theorem above is essential for handling the generation
process.
Indeed, Theorem~\ref{thm:main} will always be used in combination with
Theorem~\ref{thm:concentration}.
Most often we take $x = \Omega(\sqrt n)$ in order to have an exponential decay.
However, sometimes there is a phase transition when the central clique becomes bigger
than the remaining graph,
most notably when we look at Hamiltonian cycles in co-unipolar graphs,
and then we take $x$ to be $n/2 - \mu \sim \frac{1}{2}\log n$.
The proof of Theorem~\ref{thm:concentration} is rather technical, and is postponed to \S~\ref{sec:concentration}.
In the proof we also show that if $X \sim L(n)$, then $|\mathbb{E}X - \mu(n)| < 1$.

\subsubsection{Bipartite part}
Let $U \sim Gen(n)$ and let $(G, C)$ be the arrangement induced by the generation process.  The edges between $C$ and $\overline C$ are best seen as a random bipartite graph with colour classes $C$ and $\overline C$ and edges in between, each present with probability $1/2$ and independently of the others. The sizes of $C$ and $\overline C$ are not necessarily identical, but Theorem~\ref{thm:concentration} shows that their difference is $O(\sqrt n)$ wvhp.
Random bipartite graphs share many properties with binomial random graphs $G(n, p)$, which have been exhaustively studied.

\subsubsection{Partition part}
Let $\Pi_n$ be the set of all partitions of $[n]$.
The elements of each partition will be referred to as \emph{parts}.
We turn $\Pi_n$ into a probability space by taking each partition with equal
probability, namely $1 / |\Pi_n|$.
Let $\sigma = \sigma(n)$ be an outcome from this probability space,
i.e. $\sigma \in_u \Pi_n$.
The structure of $\sigma$ can vary from having a single part with $n$
elements to having $n$ parts with one element each.
However, $\sigma$ has a well-defined shape with high probability.

Let $r = r(n)$ be the unique root of the equation $re^r = n$
and let $B_n = |\Pi_n|$ be the $n$th Bell number.
It can be found in \cite{asymptotic_methods} that
\begin{align}
  r(n) &= \ln n - (1 + o(1))\ln \ln n, \hspace{10pt} \text{ and}\\
  \label{eq:bell_asymp}
  B_n &= \frac{1 + o(1)}{\sqrt{r}} e ^ {n(r - 1 + 1/r) - 1}.
\end{align}

There is a sizeable literature for random partitions,
see in particular~\cite{pittel_new}.
Let $|\pi|$ be the number of parts in a partition $\pi$.
If $X$ is a random variable we will use
$\widehat{X}$ to denote $(X - \mathbb{E}X) / \sqrt{\Var X}$.
One of the earliest results on random partitions is by Harper:
\begin{thm}  (Harper \cite{stirling_is_normal})
  For $\sigma \in_u \Pi_n$
  \label{thm:harper_normal}
  \begin{enumerate}
  \item
    $\widehat{|\sigma|} \overset{d}{\to} N(0, 1)$,
  \item
    $\mathbb{E}|\sigma|=
    \frac{B_{n+1}}{B_n} - 1 = (1 + o(1)) \frac{n}{r} = (1 + o(1)) \frac{n}{\ln n}$,
  \item
    $\Var|\sigma| =
    \frac{B_{n+2}}{B_n} - \left(\frac{B_{n+1}}{B_n}\right)^2 - 1
    = \frac{n}{r(r + 1)} -1 + o(1)
    = (1 + o(1)) \frac{n}{(\ln n)^2}$.
  \end{enumerate}
\end{thm}

Let $Y_t(\pi)$ be the number of parts of size $t$ of an arbitrary partition $\pi \in \Pi_n$,
so that $|\pi| = \sum_{t=1}^n Y_t(\pi)$ and $n = \sum_{t=1}^n tY_t(\pi)$.
Suppose that $\sigma \in_u \Pi_n$ from now on.

\begin{thm} (Pittel \cite{pittel_new})
  \label{thm:pittel_poisson}
  Given a partition $\pi$ and an interval $I = [k_1, k_2]$
  where $1 \le k_1 \le k_2 \le n$, let
  $$
  Y_I(\pi) = \sum_{j \in I} Y_j(\pi),\hspace{35pt}\lambda_I
  = \sum_{j \in I}\frac{r^j}{j!}.
  $$
  Then, uniformly for $I = [k_1, k_2]$ and $\xi \ge 0$,
  \[
  \mathbb{P}(|Y_I(\sigma) - \lambda_I| \ge \xi)
  = O(e^{- \min(\frac{\xi^2}{4\lambda_I}, \frac{\xi}{2})}).
  \]
\end{thm}

The last theorem is useful for small deviations,
but rather crude for deviations proportional to the mean.
A counting argument yields a better bound:

\begin{lemma}
  \label{thm:chernoff_blocks_count}
  Let $\epsilon>0$ be fixed and let $\lambda = n / r = e^r$.  Then for large $n$
  \[
  \mathbb{P}\left(\big ||\sigma| - \lambda \big| \ge \epsilon \lambda \right)
  < n e^{-n (\epsilon - \ln (\epsilon + 1))}.
  \]
  In particular, $\big ||\sigma| - \lambda \big| < \epsilon \lambda$ wvhp.
\end{lemma}
\begin{proof}
  For each $0 \le k \le n$ we have
  \begin{align*}
    \mathbb{P}(|\sigma| = k) B_n
    & \le \frac{k^n}{k!}
    \le k^n \left( \frac{k}{e} \right)^{-k}
    = \exp( n \ln k -k \ln k + k).
  \end{align*}
  Suppose that $k = (1 + c)\lambda$ for some $c = c(n) > -1$.
  Substituting in the inequality above we get:
  \begin{align*}
    \mathbb{P}(|\sigma| = (1 + c)\lambda )
    \le~ B_n^{-1} \exp(~&n \ln (1 + c)\lambda -
    (1 + c)\lambda \ln (1 + c)\lambda + (1 + c)\lambda~) \\
    =~ B_n^{-1} \exp \{~&n \ln \lambda - \lambda \ln \lambda + \lambda \\
    +~ &n \ln (1 + c) - c \lambda \ln \lambda \\
    +~ &\lambda(c - \ln (1 + c) - c \ln (1 + c))~\} \\
    \le \exp(~&n \{\ln(1+c) - c\} + \lambda \{c - \ln (1 + c) - c \ln (1 + c)\}~)
  \end{align*}
  for sufficiently large $n$.
  Here we have used $\lambda e^\lambda = n$; and that (\ref{eq:bell_asymp}) implies
  that we have
  \[
    B_n^{-1} \exp(n (r - 1 + 1/r)) < 1 \;\; \mbox{ for sufficiently large $n$}.
  \]
  Next, note that $c - \ln (1 + c) - c \ln (1 + c) \le 0$ for all $c > -1$, so we can write
  \[
  \mathbb{P}(|\sigma| = (1 + c)\lambda) \le \exp(n \{\ln(1+c) - c\})
  \]
  for large $n$.
  Let $f(x) : (-1, +\infty) \to \mathbb{R}$ be defined as $f(x) = \ln(1+x) - x$,
  so that
  \[
  \mathbb{P}(|\sigma| = (1 + c)\lambda) \le \exp(n f(c)).
  \]
  The following properties are easily verified by taking derivatives:
  $f(x) \le f(|x|)$; $f(0) = 0$; $f(x) < 0$ for $x \neq 0$;
  and $f$ is strictly increasing in $(-1, 0]$ and strictly decreasing in $[0, \infty)$.
      We deduce that for every $I \subseteq (-1, \infty]$ we have
    $\sup_{x \in I} f(x) \le f(\inf_{x \in I} |x|)$.
    Let
    $I_\delta = [0, n] \setminus
    [ \lfloor ( 1 - \delta) \lambda \rfloor, \lceil (1 + \delta) \lambda \rceil]$,
    for $\delta > 0$.
    We deduce
    \[
    \mathbb{P}\left(\big ||\sigma| - \lambda \big| \ge \epsilon \lambda \right)
    = \mathbb{P}(|\sigma| \in I_\epsilon)
    \le \max_{k \in I_\epsilon} \mathbb{P}(|\sigma| = k)
    < n\exp\{f(\epsilon)n\}. \qedhere
    \]
\end{proof}

For $\pi \in \Pi_n$ let $L(\pi)$ be the maximum size of a part.
\begin{lemma}
  \label{thm:max_block}
  Suppose $0 \le x \le n/2$ and let $\sigma \in_u \Pi_n$.
  We have
  \[
  \mathbb{P}(L(\sigma) \ge x) < (1 + o(1)) e^{-x(\ln x - \ln r - 2) + \ln n}.
  \]
\end{lemma}
\begin{proof}
  We may assume that $x \ge e^2r$ (since otherwise the RHS is $\ge 1$).
  Let $E_x$ be the event that $\sigma$ contains a part of size $\lfloor x \rfloor$.
  Then by~(\ref{eq:bell_asymp})
  and noting that $r_{n-x} - 1 + \frac{1}{r_{n-x}} \le r_n - 1 + \frac{1}{r_n}$
  (this holds because $f(t) = t - 1 + t^{-1}$ and $r_t$
  are increasing in $t \in [1, \infty)$)
    we get
  \begin{align*}
    \mathbb{P}(E_x) &\le {n \choose \lfloor x \rfloor}\frac{B_{n-\lfloor x \rfloor}}{B_n}
    \le \left(\frac{en}{x}\right)^x
    (1 + o(1))\sqrt{\frac{r_n}{r_{n-x}}} e^{-x(r_n - 1 + 1/r_n)}\\
    &\le (1 + o(1)) \left(\frac{en}{x}\right)^x \left( \frac{n}{er_n}\right)^{-x}\\
    &= (1 + o(1)) e^{-x(\ln x - \ln r -2)}.
  \end{align*}
  Thus
  \begin{align*}
    \mathbb{P}(L(\sigma) \ge x) &= \mathbb{P}\left(\bigcup_{z \ge x}E_z\right)
    \le \sum_{z \ge x} \mathbb{P}(E_z)
    < n \max_{z \ge x}\mathbb{P}(E_z)\\
    &\le(1 + o(1)) e^{-x(\ln x - \ln r - 2) + \ln n},
  \end{align*}
  since $\ln x - \ln r - 2 \ge 0$.
\end{proof}

\begin{lemma}
  Suppose $a = a(n) = o(n)$ is integral for all $n$.
  Then
  \[
  \frac{B_{n-a}}{B_n} = \exp( -a r_{n-a} + o(n)).
  \]
\end{lemma}
\begin{proof}
  We have
  \begin{align*}
    \frac{B_{n-a}}{B_n}
    &= (1 + o(1)) \exp\left((n-a)\left(r_{n-a} - 1 + \frac{1}{r_{n-a}}\right)
    -n\left(r_n - 1 + \frac 1 {r_n}\right)\right) \\
    &=\exp( (1+o(1))n\ln(1- a/n) - a (r_{n-a} - 1) + o(n)) \\
    &= \exp( (1+o(1))(-a) -a (r_{n-a} - 1) + o(n))\\
    &= \exp( -a r_{n-a} + o(n)). \qedhere
  \end{align*}
\end{proof}

\begin{cor}
  We have
  \begin{equation}
    \label{eqn.sigmamax}
    \mathbb{P}(L(\sigma) \ge n/\log n) = e^{-\Theta(n)}.
  \end{equation}
  and further, if $\epsilon(n)=o(1)$, then
  \begin{equation}
    \label{eqn.sigmamax2}
    \mathbb{P}(L(\sigma) \ge \epsilon(n) n/\log n) = e^{-o(n)}.
  \end{equation}
\end{cor}

\begin{proof}
Let $a=\lceil n/\log n\rceil$.  Then $\mathbb{P}(L(\sigma) \ge n/\log n)$ is at least $\frac{B_{n-a}}{B_n}$; and is at most $\binom{n}{a} n \frac{B_{n-a}}{B_n}$ (where the factor $n$ is to account for choosing a part (or the empty set) to add to the chosen part of size $a$).  Now we can use the last lemma.
\end{proof}

To close this section on generating random perfect graphs, let us briefly consider the graph $G$ output by the generation process, and how to recover a unipolar arrangement for $G$ or its complement (or how to seek such arrangements for the random perfect graph $P_n$).
On input an arbitrary $n$-vertex graph $H$, in $O(n^2)$-time we can test if $H$ is unipolar or co-unipolar, and if it is either then output a corresponding vertex partition.  This result is due the present authors~\cite{recog15}, 
improving on earlier $O(n^3)$-time approaches.  The method is not very complicated, but in the present case
there is a trivial $O(n^2)$-time heuristic.  

Consider $G$.
Let $A$ be the set of vertices with degree at most $n/2$, and let $B$ consist of the remaining vertices.
Then wvhp exactly one of the following two events happens:
(a) $B$ is a clique and $A$ is a disjoint union of cliques (forming a unipolar arrangement for $G$), or (b) $A$ is stable and $B$ is a disjoint union of stable sets (forming a unipolar arrangement for $\overline{G}$);
and so, since by symmetry the two events have the same probability, each holds with probability $\frac12 \pm e^{-\Omega(n)}$.
Given a unipolar arrangement of $G$ or $\overline{G}$, we can efficiently calculate graph invariants like $\alpha$,
$\omega$ and $\chi$~\cite{recog15}.

To check this, let $0<\epsilon< 1/4$.  We have seen that, in the generation process, wvhp $|k-n/2| < \epsilon n/2$ (by Theorem~\ref{thm:concentration}) and each part in $\sigma$ has size $< \epsilon n/2$ (by Lemma~\ref{thm:max_block}).  Hence, in the unipolar case ($B=1$),
by standard Chernoff bounds, wvhp each vertex in the central clique has degree in $(3n/4- \epsilon n, 3n/4+ \epsilon n)$ and each other vertex has degree in $(n/4- \epsilon n, n/4+ \epsilon n)$, so (a) holds.
Similarly, in the co-unipolar case, wvhp each vertex in the central stable set has degree in $(n/4- \epsilon n, n/4+ \epsilon n)$ and each other vertex has degree in $(3n/4- \epsilon n, 3n/4+ \epsilon n)$, so (b) holds.  Also, by~(\ref{eqn.smallcap}) wvhp 
(a) and (b) cannot both hold.


\section{Stability and Clique number}
\label{sec:stabcl}

In this section we discuss the stability number $\alpha(P_n)$ and clique number $\omega(P_n)$ of the random perfect graph $P_n$. Our main result is Theorem~\ref{thm:homogeneous}. Of course everything we say also applies to the clique covering number and chromatic number, as they equal the former two parameters for perfect graphs.

\begin{lemma}
  \label{thm:diff_alpha_sigma}
  Let $U = (1, E, (k, \sigma), \pi)$ be an arbitrary outcome of $Gen^+(n)$
  and let $G^+ = \rho(U)$.
  Then $0 \le \alpha(G^+) - |\sigma| \le 1$.
\end{lemma}

\begin{proof}
Is is easy to find a stable set of size $|\sigma|$ from the side cliques.
For the other direction, note 
that $G^+$ can be covered by $|\sigma| + 1$ cliques.
\end{proof}

The last lemma pins down $\alpha(G^+)$ to only two values,
namely $|\sigma|$ and $|\sigma| + 1$.
We next show that the second outcome is very unlikely.

\begin{lemma}
  \label{thm:stable_unipolar}
  Let $U = (1, E, (k, \sigma), \pi) \sim Gen^+(n)$, let $G^+ = \rho(U)$,
  and let $0 < \delta < 1 - \ln 2 \approx 0.3$.
  Then $\mathbb{P}(\alpha(G^+) = |\sigma| + 1) = O(e^{-n^\delta})$.
\end{lemma}

We give an informal proof first.
Condition on the induced unipolar arrangement $(G^+, C)$
and on the partition $\sigma$. We have $\alpha(G^+) = |\sigma| + 1$
if and only if there is a vertex from $C$ whose neighbourhood
does not contain any side clique in $\overline C$.
Let $A_v$ be the event that $v$ is such a vertex,
and hence $\{\alpha(G^+) = |\sigma| + 1\} \equiv \bigcup_v A_v$.
For a fixed vertex $u \in C$ and a fixed side clique $R$ we have
$\mathbb{P}(R \not\subseteq N(u)) = 1 - 2^{-|R|}$. Therefore
\[
\mathbb{P}(A_v) = \prod_{R} \left( 1 - 2^{-|R|} \right),
\]
where that product is taken over the side cliques in $\overline C$.
Up until now everything is rigorous.
Observe from Theorem~\ref{thm:pittel_poisson} that whp the number of side
cliques of size $i$ is close to $r^i / i!$,
where $r$ is the unique root of the equation $n - k = r e^r$.
For a moment assume that the number of side cliques of size $i$ equals this value,
for each $i$.  Then
\begin{align*}
  \mathbb{P}(A_v)
  &= \prod_{i=1}^n \prod_{R: |R| = i}(1 - 2^{-i})
  = \prod_{i=1}^n (1 - 2^{-i})^{r^i/i!} \\
  & \le \prod_{i=1}^n \left( e^{- 2^{-i}} \right) ^{r^i/i!} \;\;\;
  \mbox{ since } \; 1 + x \le e^x \\
  &= \exp{\left( - \sum_{i=1}^n (r/2)^i/i! \right)}
  \; \le \; \exp{\left( - e^{r/2} + 2 \right)} \\
  &= \exp{\left( - \sqrt{ (n - k) /r} + 2 \right)}
  = \exp \left( - (1 + o(1)) \sqrt{n / (2\ln n)} \right).
\end{align*}
For the last step we assumed that $k \sim n/2$,
which is easily justifiable from the concentration theorem, Theorem~\ref{thm:concentration}.
We use the union bound to finalise our arguments:
\begin{align*}
  \mathbb{P}(\alpha(G^+)
  &= |\sigma| + 1) \le \sum_v \mathbb{P}(A_v)
  \leq \exp \left( - (1 + o(1)) \sqrt{n / (2\ln n)} \right)\\
  &= O(\exp(-n^{1/2 - \epsilon})).
\end{align*}
We observe that $\delta$ from the statement is $1/2 - \epsilon$ in this case,
which may be the best possible.

\begin{proof}
  [Proof of Lemma~\ref{thm:stable_unipolar}]
Let $\mathcal{J} = [n/2 \, - 2\sqrt{n}, n/2 \, + 1 + 2\sqrt{n}]$. We have $\mathbb{P}(k \in \mathcal{J}) \ge 1 - 2^{- n}$ from Theorem~\ref{thm:concentration}.  Condition $\sigma$ on $k = k_0$ for some $k_0 \in \mathcal{J}$.  We use Theorem~\ref{thm:pittel_poisson} for $\sigma$ with $I = [0, r + r^{3/4}]$, $\lambda_I = e^r \mathbb{P}[Po(r) \le r + r^{3/4}] \le e^r$ and $\xi = 2(e^r)^{3/4}$, where $r$ is the unique root of $n - k_0 = r e^r$. The theorem says that 
\[
  \mathbb{P}(Y_I(\sigma) \ge \lambda_I - \xi \big | k = k_0)
  \ge 1 - c \exp{(-\sqrt{e^r})},
\]
for some constant $c$. In order to give a lower bound on $\lambda_I$, we use a standard Chernoff bound for a Poisson random variable:
$\mathbb{P}(Po(r) \ge (1 + \eta) r) < \exp{(-r \eta^2/4)}$, for $0 < \eta \le 1$. Taking $\eta = r^{-1/4}$ we obtain:
\[
  \lambda_I - \xi > e^r (1 - \exp{(-\sqrt{r}/4)} - 2 (e^r)^{-1/4}) = e^r(1 - o(1)).
\]
Define the events $A_v$ as in the preceding informal discussion. The event \{$Y_I(\sigma) \ge \lambda_I - \xi$\} implies that
  in $\sigma$ there are at least $e^r(1 + o(1))$ parts of size at most $r(1 + o(1))$,
  hence
  \begin{align*}
    \mathbb{P}&(A_v  \big| (Y_I(\sigma) \ge \lambda_I - \xi) \cap (k = k_0))\\
    &\le (1 - 2^{-r(1 + o(1))})^{e^r(1 + o(1))} \\
    &\le \exp{\{- 2^{-r(1 + o(1))} e^r(1 + o(1))\}} \\
    &= \exp{\{-(e/2)^{r(1 + o(1))}\}} \\
    &= \exp{\{-\exp{[r \cdot (1 - \ln 2)(1+ o(1)) ]}\}}\\
    & = \exp{\{
      \left((n - k_0) / r \right)
      ^ { (1 - \ln 2)(1+ o(1))}
      \}}   \;\;\; \mbox{ since } \; e^r = (n-k_0)/r \\
    &=O(\exp {(-n^\delta)}).
  \end{align*}
  From the union bound we get
  \begin{align*}
    \mathbb{P}(\alpha(G^+) = |\sigma| + 1 \big| k = k_0)
    &\le \mathbb{P}(Y_I(\sigma) < \lambda_I - \xi \big | k = k_0)\\
    &+ \sum_v \mathbb{P}(A_v  \big| (Y_I(\sigma) \ge \lambda_I - \xi) \cap (k = k_0))\\
    &=O(\exp {(-n^\delta)}).
  \end{align*}
  Hence
  \begin{align*}
    \mathbb{P}(\alpha(G^+) = |\sigma| + 1)
    &\le \mathbb{P}(k \not\in \mathcal{J}) +
    \max_{k_0 \in \mathcal{J}} \mathbb{P}(\alpha(G^+) = |\sigma| \big | k = k_0) \\
    &=O(\exp {(-n^\delta)}). \qedhere
  \end{align*}
\end{proof}

Note that the bound for $\delta$ in the last proof is $1 - \ln 2 \approx 0.3$,
which is slightly worse than the constant in the informal discussion.

We shorten $G^+ \sim \rho(Gen^+(n))$ to $G^+ \sim Gen^+(n)$.
\begin{prop}
  \label{thm:stable_normal}
  Let $G^+ \sim Gen^+(n)$, and denote $\alpha(G^+)$ by $X$.  Then, as $n \rightarrow \infty$,
  $\widehat{X} \overset{d}{\to} N(0, 1)$,
  $\mathbb{E}[X] \sim \frac{n}{2\ln n}$, and
  $\Var[X]  \sim \frac{n}{2\ln^2 n}$.
\end{prop}
We write $Y$ for $|\sigma|$ and $Y_m$ for $|\sigma_m|$ when $\sigma_m \in_u \Pi_m$.
To prove the last result we will use the following claim.
\begin{claim}
  $\mathbb{E}Y_{n+1} = \mathbb{E}Y_n + \frac{1 + o(1)}{\ln n}$.
\end{claim}
\begin{proof}
  From Theorem~\ref{thm:harper_normal},
  $\mathbb{E}Y_n = \frac{B_{n+1}}{B_n} - 1$
  and $\Var Y_n = \frac{B_{n+2}}{B_n} - \left(\frac{B_{n+1}}{B_n}\right)^2 - 1$.
  Hence
  \[
  \Var Y_n
  = \frac{B_{n+2}}{B_{n+1}}\frac{B_{n+1}}{B_n} - \left(\frac{B_{n+1}}{B_n}\right)^2 - 1
  = (\mathbb{E} Y_{n+1} + 1)(\mathbb{E}Y_n + 1) - (\mathbb{E}Y_n + 1)^2 - 1.
  \]
  Therefore
  \begin{align*}
    (\mathbb{E}Y_{n+1} + 1) - (\mathbb{E}Y_n + 1)
    &= \frac{\Var Y_n + 1}{\mathbb{E}Y_n + 1}
    = \frac{1 + o(1)}{\ln n}.
    \qedhere
  \end{align*}
\end{proof}
\begin{proof}
  [Proof of Proposition~\ref{thm:stable_normal}]
  Recall that $\sigma$ is a random partition created by first choosing a value
  for $k \sim L(n)$, and then taking $\sigma \in_u \Pi_{n-k}$.
  It suffices to show that $\widehat{Y} \overset{d}{\to} N(0, 1)$.
  Indeed, if this were the case we get
  $Y \leq X \leq Y+1, \Var(Y) \to \infty, \widehat{Y} \overset{d}{\to} N(0, 1)$,
  and hence $\widehat{X} \overset{d}{\to} N(0, 1)$.

  Let $\Phi(x)$ be the CDF of $N(0, 1)$.
  For an arbitrary fixed $x$ and $\epsilon>0$ we show that for large $n$
  \[
  |\mathbb{P}(\widehat{Y} < x) - \Phi(x)| < \epsilon.
  \]

  Find $\ell = \ell(\epsilon)$ such that
  $\mathbb{P}(k \not\in \mathcal{I}) < \epsilon / 2 + n^{-n}$
  for $\mathcal{I} = \mathcal{I}(n) =$
  $[\lfloor \mathbb{E}k - \ell \rfloor, \lceil \mathbb{E}k + \ell\rceil]$.
  Such $\ell$ exists by Theorem~\ref{thm:concentration}.
  We emphasise that the length of $\mathcal{I}$ is at most $2\ell + 1$,
  which is independent of $n$.
  \begin{align*}
    \left| \mathbb{P}(\widehat{Y} < x) - \Phi(x) \right|
    &= \sum_{c = 0}^n
    \left| (\mathbb{P}(\widehat{Y} < x \big| k = c)
    - \Phi(x)) \, \mathbb{P}(k = c) \right|\\
    &< \sum_{c \in \mathcal{I}}
    \left| (\mathbb{P}(\widehat{Y} < x \big| k = c)
    - \Phi(x)) \right| \mathbb{P}(k = c) + \epsilon / 2 + n^{-n}\\
    &\le \max_{c \in \mathcal{I}}
    \left| \mathbb{P}(\widehat{Y} < x \big| k = c) - \Phi(x) \right|
    + \epsilon / 2 + n^{-n}.
  \end{align*}
Fix any integer valued function $q(n)$ with $q(n) \in \mathcal{I}(n)$.
We shorten $q(n)$ to $q$ for readability.
\begin{align}
    \big| \mathbb{E}Y - \mathbb{E}Y_{n - q} \big|
    &= \big| \mathbb{E} \left\{\mathbb{E} \left[Y \big| k\right] \right\}
    - \mathbb{E}Y_{n - q} \big| \nonumber \\
    &\le \sum_{i=0}^n
    \big| \mathbb{E}Y_{n-i} - \mathbb{E}Y_{n - q} \big| \, \mathbb{P}(k = i) \nonumber \\
    &\le \sum_{i \in \mathcal{I}}
    \big| \mathbb{E}Y_{n-i} - \mathbb{E}Y_{n - q} \big| \, \mathbb{P}(k = i)
    + \epsilon/2 + n^{-n} \nonumber \\
    &\le \max_{i \in \mathcal{I}}
    \big| \mathbb{E}Y_{n-i} - \mathbb{E}Y_{n - q} \big| + \epsilon/2 + n^{-n} \nonumber \\
    &\le (1 + o(1))\frac{2\ell + 1}{2\ln n} + \epsilon/2 + n^{-n} \nonumber \\
    &= \epsilon/2 + o(1) \label{eq:diff_exp}.
\end{align}
We proceed with the variance in a similar fashion.
\begin{equation*}
    \big| \Var Y - \Var Y_{n - q} \big|
    \le \max_{i \in \mathcal{I}}
    \big| \Var Y_{n-i} - \Var Y_{n - q} \big| + \epsilon/2 + n^{-n}
    = o(\Var Y_{n - q}).
\end{equation*}
For the last equality we have used part 3 of Theorem~\ref{thm:harper_normal}. We conclude
\begin{equation}
    \label{eq:diff_var}
    \Var Y = (1 + o(1))\, \Var Y_{n - q}.
\end{equation}
Write  $y = (x (\Var Y)^{1/2} + \mathbb{E}Y - \mathbb{E}Y_{n - q}) (\Var Y_{n - q})^{-1/2}$, so that  
\[ \mathbb{P}(\widehat{Y} < x \big| k = q) = \mathbb{P}(\widehat{Y_{n - q}} < y).\] 
By Theorem~\ref{thm:harper_normal} part 3, we have $\Var Y \to \infty$ as $n \to \infty$; and so by (\ref{eq:diff_exp}) and (\ref{eq:diff_var}) 
it follows that  $\lim_{n \to \infty} y = x$.  Find $\delta > 0$, so that $\Phi(x) \le \Phi(x - \delta) + \epsilon/6$. From Theorem~\ref{thm:harper_normal} we have
$\mathbb{P}(\widehat{Y_{n-q}} < x - \delta) \to \Phi(x - \delta)$
if $n - q \rightarrow \infty$. For large $n$
\[
  \mathbb{P}(\widehat{Y_{n-q}} < y)
  \ge \mathbb{P}(\widehat{Y_{n-q}} < x - \delta)
  \ge \Phi(x - \delta) - \epsilon / 6
  \ge \Phi(x) - \epsilon / 3.
\]
We can do the same in the other direction and deduce that for large $n$
\[
  |\mathbb{P}(\widehat{Y} < x \big| k = q) - \Phi(x)| < \epsilon / 3.
\]
But $q$ was an arbitrary function, hence for large $n$
\[
  \max_{c \in \mathcal{I}}
  \left| \mathbb{P}(\widehat{Y} < x \big| k = c) - \Phi(x) \right| < \epsilon/3,
\]
and the proposition is proven.
\end{proof}

We continue with $\omega(G^+)$.
The overview is that
usually there is a unique clique of maximum size in $G^+$ -- the central clique.
The vertices in the central clique have degree $(1 + o(1))3n/4$ whp,
while the vertices outside have degree $(1 + o(1))n/4$ whp,
so the central clique $C$ from the generation is easily recognisable
in the resulting graph (and in the 4-tuple $U$, as well as $k$ and $B$,
we know $\sigma$ up to relabelling of $\overline{C}$,
we know a little about $E$ and we know $\pi([k])=C$).

\begin{lemma}
  \label{thm:cliqe_unipolar}
  Suppose $G^+ = \rho(U)$, where $U = (B, E, (k, \sigma), \pi) \sim Gen^+(n)$.
  Then both the probabilty that $C$ is not the unique clique of maximum
  size and $\mathbb{P}(\omega(G^+) \neq k)$ are $2^{-(\frac{1}{2} + o(1))n}$.
\end{lemma}
\begin{proof}
  As usual, let $\mathcal{I} = [n/2 - \sqrt{n}, n/2 + \sqrt{n}]$,
  and recall that $k \in \mathcal{I}$ wvhp by Theorem~\ref{thm:concentration}.
  Let $k_0 \in \mathcal{I}$, and condition on $k = k_0$.

  We first give a lower bound for $\mathbb{P}(\omega(G^+) \neq k_0)$.
  With probability at least $2^{-k_0}$ there is a vertex outside
  the central clique which is adjacent to all vertices from the central clique,
  and hence $\mathbb{P}(\omega(G^+) > k_0) \ge 2^{-(\frac{1}{2} + o(1))n}$.

We now prove an upper bound.
Of course,  $\mathbb{P}(\omega(G^+) \neq k)$ is at most the probabilty that $C$ is not the unique clique of maximum size:
we upper bound the latter probability. 
Let $(G^+, C)$ be the unipolar arrangement induced by $U$. There are two cases in which $G^+$ contains a clique $M \neq C$ of size at least $k_0$.
The first case is when $|\overline C \cap M| = 1$, when also 
$|C \setminus M| \leq 1$.
The probability of that happening is at most
$(n -k_0) k_0 2^{-(k_0-1)} = 2^{-(\frac{1}{2} + o(1))n}$.

  Let us focus on the second case; $|\overline C \cap M| > 1$.
  This implies that there are at least two vertices from $\overline C$
  with degree at least $k_0$.
  If we substitute $n/\sqrt{\log n}$ for $x$ in Lemma~\ref{thm:max_block},
  we get that the probability that there is a part in $\sigma$ of size more than
  $n/\sqrt{\log n}$ is less than $e^{-n}$ for large $n$.
  Assume there is no such part in $\sigma$.

  For every $v \in \overline C$ write $d(v) = d_C(v) + d_{\overline C}(v)$.
  We have $d_C(v) \sim Bin(k_0, 1/2)$
  and $d_{\overline C}(v) \le n/\sqrt{\log n} - 1$.
  Let $E_v$ be the event $\{d_C(v) \ge k_0 - n/\sqrt{\log n}\}$.
  Next we use the Chernoff bound in the form:
  $\mathbb{P}(X  \ge mp + t) \le e^{-2t^2/m}$,
  for $X \sim Bin(m, p)$ and $t \geq 0$.
  We obtain
  \[
  \mathbb{P}(E_v) \le e^{-(1 + o(1))k_0/2}.
  \]
  The events $E_v$ and $E_u$ are independent for $u \neq v$,
  hence $\mathbb{P}(E_v \cap E_u) \le e^{-(1 + o(1))k_0}$
  for every pair $u \neq v$.
  We deduce that the probability that there are two vertices from
  $\overline C$ with degree at least $k_0$ is at most
  $n^2e^{-(1 + o(1))k_0} + e^{-n} = e^{-(\frac{1}{2} + o(1))n}$.

  Combining the two  cases completes the proof of the upper bound.
\end{proof}

\begin{lemma}
  \label{thm:intersect}
  For $G^+ \sim Gen^+(n)$ and a fixed $\epsilon > 0$ we have wvhp
  \begin{align*}
    (1 - \epsilon)\frac{n}{2\ln n} ~< ~&\alpha(G^+) ~< ~(1 + \epsilon)\frac{n}{2\ln n}
    \mbox{ , and}\\
    \frac{n}{2} - \epsilon \sqrt n ~< ~&\omega(G^+) ~< ~\frac{n}{2} + \epsilon \sqrt n.
  \end{align*}
  In particular, $\alpha(G^+) < \omega(G^+)$ wvhp.
\end{lemma}
\begin{proof}
  The first line follows from Lemma~\ref{thm:diff_alpha_sigma},
  Lemma~\ref{thm:chernoff_blocks_count} and Theorem~\ref{thm:concentration};
  the second line follows from Theorem~\ref{thm:concentration}
  and Lemma~\ref{thm:cliqe_unipolar}.
\end{proof}
\noindent
Observe that the proof of Theorem~\ref{thm:concentration} will 
complete the proof of Lemma~\ref{thm:intersect}.
For every graph $G$ we have $\omega(G) = \alpha(\overline{G})$
and therefore everything we showed for $\alpha(G^+)$ and $\omega(G^+)$
directly translates to $\omega(G^-)$ and $\alpha(G^-)$.
Let $h(G) = \min \{\alpha(G), \omega(G)\}$ and $H(G) = \max \{\alpha(G), \omega(G)\}$
for every graph $G$.
(Here $h$ and $H$ come from \underline{h}omogeneous.)
\bigskip

\begin{thm}
  \label{thm:homogeneous}
  For $P_n \in_u \mathcal{P}_n$,
  \begin{enumerate}
  \item
    $h(P_n)$ is asymptotically normal
    with mean $\sim \frac{n}{2\ln n}$
    and variance $\sim \frac{n}{2\ln^2 n}$,
  \item
    $d_{TV}(H(P_n), L(n)) = e^{-\Omega(n)}$.
  \end{enumerate}
\end{thm}

\begin{proof}
  Let $U = (B, E, (k, \sigma), \pi) \sim Gen(n)$.
  We couple $G_n$, $G_n^+$ and $G_n^-$ by writing $U^+ = (1, E, (k, \sigma), \pi)$,
  $U^- = (-1, E, (k, \sigma), \pi)$, $G_n = \rho(U)$, $G_n^+ = \rho(U^+)$
  and $G_n^- = \rho(U^-)$.
  We have
\begin{align*}
    d_{TV}(h(P_n), \alpha(G_n^+))
    &\le d_{TV}(h(P_n), h(G_n)) + d_{TV}(h(G_n), \alpha(G_n^+))\\
    &\le d_{TV}(\mathbb{P}_{\mathcal{P}_n}, \mathbb{P}_{Gen(n)})
    + \mathbb{P}(h(G_n) \neq \alpha(G_n^+)).
\end{align*}
But $d_{TV}(\mathbb{P}_{\mathcal{P}_n}, \mathbb{P}_{Gen(n)}) = e^{-\Omega(n)}$
from Theorem~\ref{thm:main}; and, since $h(G_n)=h(G^+_n)$,
\[
\mathbb{P}(h(G_n) \neq \alpha(G_n^+))
= \mathbb{P}(\omega(G_n^+) < \alpha(G_n^+))
= e^{-\Omega(n)}\] from Lemma~\ref{thm:intersect}.
Thus $d_{TV}(h(P_n), \alpha(G^+)) = e^{-\Omega(n)}$.
The first part in the statement now follows from
Proposition~\ref{thm:stable_normal}.

Similar arguments establish the second part of the theorem.
Much as above we may write
\[
d_{TV}(H(P_n), L(n)) \le d_{TV}(H(P_n), H(G_n)) + d_{TV}(H(G_n),
\omega(G_n^+)) +  d_{TV}(\omega(G_n^+), L(n)).
\]
The first and third terms in the upper bound here are $e^{-\Omega(n)}$
by Theorem~\ref{thm:main} and Lemma~\ref{thm:cliqe_unipolar}.
Since $H(G_n)=H(G^+_n)$, the second term is at most
$\mathbb{P}(\alpha(G_n^+) > \omega(G_n^+))$,
which is $e^{-\Omega(n)}$ by Lemma~\ref{thm:intersect}.
Thus $d_{TV}(H(P_n), L(n)) = e^{-\Omega(n)}$, as required.
\end{proof}

The last result
shows that the distribution $L(n)$ is not only needed for the approximate generation process,
but also corresponds to an important property of perfect graphs.
\medskip

Finally here in this section on stability and clique numbers, let us briefly consider algorithmic aspects, following the comments at the end of the last section, and using Lemma~\ref{thm:intersect}.  

Given a random perfect graph $P_n$,
let $A$ be the set of vertices with degree at most $n/2$, and let $B$ consist of the remaining vertices.   Then wvhp exactly one of the following two cases holds. 
\begin{description}
\item{(a)}
$B$ is a clique and indeed is the unique maximum clique (by Lemma~\ref{thm:cliqe_unipolar}), so $\omega=|B|$; and $A$ is a disjoint union of cliques, so $\alpha$ is either the number of parts (cliques) in $A$ or this number plus 1 (by Lemma~\ref{thm:diff_alpha_sigma}), and we can easily tell which in $O(n^2)$-time. 
\item{(b)} 
$A$ is the unique maximum stable set, so $\alpha=|A|$; and $B$ induces a complete multipartite graph, so $\omega$ is either the number of parts or this number plus 1, and we can easily tell which in $O(n^2)$-time. 
\end{description}
In the first case, to tell if $\alpha$ is the number of parts plus 1 we just test if some vertex in $B$ is non-adjacent to at least one vertex in each side clique; and similarly for the second case.  In particular, wvhp in $O(n^2)$ time we can determine $\omega$ and $\alpha$ (and know they are correct).


\section{Induced Subgraphs}
\label{sec:subgraphs}
Suppose $P_n \in_u \mathcal{P}_n$.
From Lemma~\ref{thm:intersect} and Theorem~\ref{thm:main} we have
$\omega(P_n) \ge \frac{(1-\epsilon)n}{2\ln n}$ wvhp,
so we can of course find any fixed graph $H$ as a subgraph of $P_n$ wvhp.

We use  $H \induce G$ to denote that $H$ is an induced subgraph of $G$.
We observed already that the generated graphs, $\rho(Gen(n))$,
are symmetric with respect to taking complements.
The same is true for perfect graphs, that is  $P_n \sim \overline{P_n}$, by
the weak perfect graph theorem~\cite{weak_perfect1, weak_perfect2}.
Thus we see that, if $H$ is fixed and $P_n \in_u \mathcal{P}_n$, then
\begin{align}
  \label{eq:induced_symmetry}
  \mathbb{P}(H \induce P_n) = \mathbb{P}(\overline{H} \induce P_n).
\end{align}

\begin{lemma}
  \label{thm:elaborate_subgraph}
  Suppose that $G^+_n \sim Gen^+(n)$ and that $H$ is fixed.  Then
  $$
  \mathbb{P}(H \induce G^+_n) =
  \begin{cases}
    0 &\mbox{ if } H \notin \mathcal{GS}^+, \\
    1 - e^{-\Omega(n \ln n)} &\mbox{ if } H \in    \mathcal{GS}^+.
  \end{cases}
  $$
\end{lemma}

To prove this lemma, we first establish two preliminary claims.
\begin{claim}
  Suppose $\sigma \in_u \Pi_n$,
  and let $a = a(n) \le n^\delta$ for some $0<\delta <1$.
  The probability that $\sigma$ contains at most $a$ parts is
  at most $\exp\{-(1-\delta+o(1))n \ln n\}$.
\end{claim}
\begin{proof}
  The probability of the event in the statement is
  at most $a^n/B_n$, and the bound follows,
  since $\ln B_n = (1+o(1)) n \ln n$.
\end{proof}

\begin{claim}
  Suppose $\sigma \in_u \Pi_n$,
  and let $l \ge 1$ be a fixed integer.
  The probability that all parts of $\sigma$ are of size at most $l$
  is at most $\exp\{ - (\frac 1 l + o(1)) n \ln n\}$.
\end{claim}
\begin{proof}
  We call a finite sequence $(a_i)_{i=1}^k$ \emph{type}
  if $\sum_{i=1}^k a_i = n$.
  We say that $\sigma$ has type $(a_i)$ if
  $(a_i)$ corresponds to the sizes of the parts of $\sigma$
  taken in any order.
  The number of partitions with type $(a_i)$ with $a_i \le l$ for all $i$ is
  \[
  \frac{n!}{k! \prod_{i=1}^{k}a_i!} \le \frac{n!}{k!}
  \le \frac{e \sqrt n}{\sqrt{2\pi k}} \cdot \frac{n^n e^{-n}}{k^k e^{-k}}
  = \frac{n^n}{(\frac n l)^{\frac n l}}2^{O(n)}
  = \exp \left \{(1 - \frac 1 l)n \ln n + O(n)\right\}.
  \]
  The number of types $(a_i)$ with $a_i \le l$ for all $i$
  is at most $l^n = \exp(O(n))$,
  and hence the number of partitions with parts at most $l$
  is $\exp \left \{(1 - \frac 1 l)n \ln n + O(n)\right\}$.
  The statement in the lemma follows after dividing by $B_n$.
\end{proof}

\begin{cor}
  \label{thm:good_partitions}
  For fixed $\delta > 0$ and $l \ge 1$,
  the probability that a random partition $\sigma \in_u \Pi_n$
  contains at most $n^\delta$ parts of size more than $l$
  is $\exp(-\Theta(n \ln n))$.
\end{cor}
\begin{proof}
  Let $\Pi_Q^\prime$ and $\Pi_Q''$ be the classes of partitions of $Q$
  with at most $n^\delta$ parts and with parts of size at most $l$
  respectively.
  The number of partitions described in the statement can be bounded by
  $\sum_{S \subseteq [n]}|\Pi_S'||\Pi_{\overline{S}}''| = \exp(-O(n \ln n))$
  by the previous two claims.
  For a lower bound consider the partition $\pi$ with $n$ parts;
  $\mathbb{P}(\sigma = \pi) = B_n^{-1} = \exp(-\Theta(n \ln n))$.
\end{proof}

\begin{proof}[Proof of Lemma~\ref{thm:elaborate_subgraph}]
  Since $\mathcal{GS}^+$ is a hereditary class of graphs
  and $G^+_n$ takes as values only $\mathcal{GS}^+$-graphs with non-zero probability,
  we get that if $H \notin \mathcal{GS}^+$,
  then $\mathbb{P}(H \induce G^+_n) = 0$ for any $n$.

  Now suppose that $H \in \mathcal{GS}^+$.
  Fix a unipolar arrangement $(H, A)$.
  Suppose that $\pi \in \Pi_{\overline A}$ corresponds to $H[\overline A]$,
  and let $l = L(\pi)$ be the maximum size of a part in $\pi$.
  Let $I_n$ be the interval $[n/2 - n^{2/3}, n/2 + n^{2/3}]$.
  Recall that we use $k$ to denote the size of the central clique
  in the generation process.
  By Theorem~\ref{thm:concentration},
  $\mathbb{P}(k \not\in I_n) \leq 2n^{-n}$ for $n$ sufficiently large;
  so we may condition on $k \in I_n$.
  Let $(G_n^+, C)$ be the unipolar arrangement induced by
  the generation process.

By Corollary~\ref{thm:good_partitions} (with $1/2 <\delta<1$) there are at least $| \pi| n^{1/2}$ side 
 cliques in $\overline C$ of size at least $l$ with probability $1 - \exp(-\Theta(n \ln n))$, so we may condition on this event. Hence we can find a collection of at least $t_1 = \lceil n^{1/2} \rceil$ 
disjoint sets $T_j \subseteq \overline C$ such that $G^+_n[T_j] \cong H[\overline A]$.  Pick $t_0 = \lceil n/(3|A|) \rceil $ disjoint subsets $S_i$ of $C$, each of size $|A|$. Then $G^+_n[S_i] \cong H[A]$ for every $i$, and hence for $1 \le i \le t_0$ and $1 \le j \le t_1$ we have
$p_0 = \mathbb P(G^+_n[S_i \cup T_j] \cong H)  \geq  2^{-|A||\overline A|}$.  
  The probability that no such pair induces $H$ is
  $(1 - p_0)^{t_0 t_1} = \exp(-\Theta(n^{3/2}))$.
  The failure probability is dominated by the term $\exp(-\Theta(n \ln n))$.
\end{proof}
\bigskip

\noindent
The bound in Lemma~\ref{thm:elaborate_subgraph} cannot be improved.
For suppose $H$ is the 4-edge path: then
\[
\mathbb{P}(H \not \subseteq_i G_n^+) \geq
\mathbb{P}(\sigma \mbox{ has only parts of size }1)
\geq 1/B_n = e^{- O(n \ln n).}
\]
It is now immediate that for a co-unipolar graph $H$ we have
\[
\mathbb{P}(H \subseteq_i G_n^-) = \mathbb{P}(\overline H \subseteq_i G_n^+)
= 1 - e^{\Omega(n \ln n)},
\]
and hence we arrive at the following theorem:

\begin{thm}
  \label{thm:subgraph_lemma}
  Let $H$ be any (fixed) graph, and let $P_n \in_u \mathcal{P}_n$.  Then
  \[
  \mathbb{P}(H \induce P_n) =
  \begin{cases}
    e^{-\Omega(n)}                    &\mbox{ if } H \notin \mathcal{GS}, \\
    1/2 \pm e^{-\Omega(n)}
    &\mbox{ if } H \in    \mathcal{GS} \setminus (\mathcal{GS}^+ \cap \mathcal{GS}^-), \\
    1   - e^{-\Omega(n)} &\mbox{ if } H \in      (\mathcal{GS}^+ \cap \mathcal{GS}^-).
  \end{cases}
  \]
\end{thm}

\begin{proof}
  Consider $Q_n = \{G \in \mathcal{G}_n: H \induce G\}$
  and $\mathbb{P}_{\mathcal{P}_n}(Q_n)$.
  Now the result follows from Lemma~\ref{thm:elaborate_subgraph},
  Equation~(\ref{eq:partition_of_gs}) and Theorem~\ref{thm:main}.
\end{proof}

Note that Theorem~\ref{thm:subgraph_lemma} accords with~(\ref{eq:induced_symmetry}).
It also follows from Theorem~\ref{thm:subgraph_lemma} that for every graph $H$:
\[
\mathbb{P}\left( (H \induce P_n) \bigcup (\overline{H} \induce P_n) \right) =
\begin{cases}
  e^{-\Omega(n)}   &\mbox{ if } H \notin \mathcal{GS}, \\
  1 - e^{-\Omega(n)}  &\mbox{ if } H \in    \mathcal{GS}.
\end{cases}
\]
and
\[
\mathbb{P}\left((H \induce P_n) \bigcap (\overline{H} \induce P_n) \right) =
\begin{cases}
  e^{-\Omega(n)}   &\mbox{ if } H \notin \mathcal{GS}^+ \cap \mathcal{GS}^-, \\
  1 - e^{-\Omega(n)}  &\mbox{ if } H \in    \mathcal{GS}^+ \cap \mathcal{GS}^-.
\end{cases}
\]


\section{Clique Colouring}

\label{sec:clique}
A $j$-\emph{clique colouring} of a graph is a colouring of the vertices with $j$ colours
so that no maximal clique is monochromatic (ignoring any isolated vertices).
We improve Corollary~6 of \cite{cliques}, which states that almost all perfect
graphs are $3$-clique-colourable.

\begin{thm}
  \label{thm:cliques}
  Almost all perfect graphs are $2$-clique-colourable.
\end{thm}

Corollary~6 of \cite{cliques} is proved by showing that all generalised split graphs
are $3$-clique-colourable,
and then using the theorem of Pr\"omel and Steger discussed earlier.
In the same article it is shown that there are generalised split graphs
which are not $2$-clique-colourable.
There are several other subclasses of perfect graphs for which
the clique-chromatic number is known to be at most 3,
and it was conjectured in~\cite{cliques_conjecture} that perfect graphs
have bounded clique-chromatic number.
This was recently disproved in~\cite{cliques_disproff}.
See~\cite{colin_clique_colouring} for recent work on clique-colourings
of binomial random graphs and of geometric graphs.
We prove Theorem~\ref{thm:cliques} using Theorem~\ref{thm:main}.
We consider unipolar and co-unipolar graphs separately.

\subsection{Clique colouring of unipolar graphs}

We start with a deterministic lemma.
We say that a vertex \emph{sees} a set of vertices if it has a neighbour from the set.
\begin{lemma} \label{thm:cliques_deterministic}
  Let $G$ be a unipolar graph, with given central clique $C$ and thus given side cliques.
  If there is a vertex $x \in C$ which sees each side clique which is a maximal clique in $G$,
  then $G$ is $2$-clique-colourable.
\end{lemma}

\begin{proof}
  We may suppose that $G$ contains at least $2$ vertices.
  Let $x$ be the vertex of $C$ that sees each maximal side clique.
  Choose a vertex $y_Q \in N(x) \cap Q$ for each side clique $Q$
  with $N(x) \cap Q \neq \emptyset$.
  Define a colouring $c : V(G) \rightarrow \{1, 2\}$ in the following way:
  $$
  c(v) =
  \begin{cases}
    1 & \mbox{if } v = x,\\
    2 & \mbox{if } v \in C - x,\\
    2 & \mbox{if } v = y_Q \mbox{ for some side clique } Q,\\
    1 & \mbox{otherwise}.
  \end{cases}
  $$

  We claim that $c$ is a proper clique colouring.
  Note first that $N(x) \neq \emptyset$: this is trivial if $|C| \geq 2$;
  and if $C=\{x\}$ then there is a side clique $Q$,
  and $Q$ must contain a neighbour of $x$ (whether it is a maximal clique or not).

Suppose for a contradiction that $M$ is a monochromatic maximal clique.
Since $M$ is a clique, it must lie in the co-bipartite graph induced by $C$
and a side clique $Q$.
Further $M \not\subseteq C$: for if so then $M=C$, and $C$ is not monochromatic
if $|C| \geq 2$ and $C$ is not a maximal clique if $C=\{x\}$.

If $N(x) \cap Q$ is empty, then $M$ contains a vertex of colour $1$
in $Q$, does not contain $x$, and since $M$ is monochromatic,
$M$ is contained in $Q$.
By the definition of $x$, $Q$ is not maximal;
hence $M$ is not maximal, but this is a contradiction.

Thus we may assume that $N(x) \cap Q$ is nonempty,  and so contains $y_Q$.
Then either (1) $M \subseteq \{x\} \cup (Q \setminus \{y_Q\})$ or
(2) $M \subseteq (C \setminus \{x\}) \cup \{y_Q\}$.
In case (1), we could add $y_Q$ to $M$, and in case (2) we could add $x$.
Thus again $M$ is not maximal, and we have our final contradiction.
\end{proof}

Let $t=t(n)= \log n - 2 \log\log n$
($t$ is for threshold).
Given a unipolar arrangement $(G, C)$ of order $n$, we use $t$ to split the side cliques of $(G, C)$ into two categories - the \emph{big} side cliques with at least $t$ vertices, and the \emph{small} side cliques with less than $t$ vertices.  We shall see in the next two lemmas that, for $G^+ \sim Gen^+(n)$, whp the big side cliques can all be seen by some $x \in C$ and the small side cliques are not maximal, so we can apply the above deterministic lemma to deduce that whp $G^+$ is 2-clique-colourable.
\begin{lemma}
Let $G^+ \sim Gen^+(n)$, with induced unipolar arrangement $(G,C)$.
Then with probability $1-2^{-\Omega(\log^2 n)}$ there is a vertex $x \in C$ which sees each big side clique. 
\end{lemma}
\begin{proof}
Condition on $k =k_0 \in \mathcal{I} = [n/2 - \sqrt{n}, n/2 + \sqrt{n}]$, as usual. Condition also on the partition $\sigma$ of $W=[n] \setminus [k]$; and assume wlog that the permutation $\pi$ is the identity, so $C=[k]$.
Let $\mathcal B$ be the set of all big side cliques, so $|\mathcal B| \le (n - k_0) / t$.

For $v \in C$ and $Q \in \mathcal B$, the probability that $v$ does not see $Q$ is
  \[ 2^{-|Q|} \le 2^{-t} = \frac{\log^2 n}{n} . \] 
Thus the probability that $v$ sees each $Q \in {\mathcal B}$ is at least
\begin{eqnarray*}
   \left(1 - \frac{\log^2 n}{n}\right)^{\frac{n\!-\!k_0}{t}} & = & 
  \exp \left(-\frac{\log^2 n}{n} \frac{n\!-\!k_0}{t} +O(\frac{\log^3 n}{n}) \right)\\
  & = & \exp\left(- (\frac12+o(1)) \log n \right) = n^{-\alpha+o(1)}
\end{eqnarray*}
where $\alpha = 1/(2 \log e) \approx 0.35$.
Then, the probability that for each vertex $v \in C$ there is some $Q \in \mathcal B$ which $v$ fails to see is at most
\[ (1- n^{-\alpha+o(1)})^{k_0} \leq e^{- n^{\frac12 - \alpha +o(1)}},\]
which is easily at most $2^{-\Omega(\log^2 n)}$.
\end{proof}

\begin{lemma}
  Let $G^+ \sim Gen^+(n)$, with induced unipolar arrangement $(G,C)$.  Then with probability $1-2^{-\Omega(\log^2 n)}$ no small side clique is a maximal clique.
\end{lemma}
\begin{proof}
Condition on $k =k_0 \in \mathcal{I} = [n/2 - \sqrt{n}, n/2 + \sqrt{n}]$, as usual.
Let us say that a vertex $v \in C$ \emph{extends} a side clique $Q$ if $Q \subseteq N(v)$. For $v \in C$ and $Q \in \mathcal{B}$, the probability that $v$ does not extend $Q$ is 
  \[ 1 - 2^{-|Q|} \le 1 - 2^{-t} = 1 - \frac{\log^2 n}{n}
  \le e^{-\frac{\log^2 n}{n}}.\]
Thus the probability that no $v \in C$ extends $Q$ is at most
  \[ \left(e^{-\frac{\log^2 n}{n}}\right)^{k_0} = e^{-(\frac12+o(1))\log^2 n}.  \]
But there are at most $n$ small side cliques, so the same upper bound holds for the probability that some small side clique is a maximal clique.
\end{proof}

\begin{lemma}
  \label{thm:cliques_unipolar}
$G^+ \sim Gen^+(n)$ is $2$-clique-colourable with probability $1-2^{-\Omega(\log^2 n)}$.
\end{lemma}
\begin{proof}
  This follows from the last three lemmas.
\end{proof}

\subsection{Clique colouring of co-unipolar and perfect graphs}
Co-unipolar graphs are easier to clique-colour than unipolar graphs.

\begin{lemma}
  \label{thm:cliques_co-unipolar}
  The random graph $G^- \sim Gen^-(n)$ is $2$-clique-colourable wvhp.
\end{lemma}
\begin{proof}
  We first prove an auxiliary claim.
\begin{claim}
    Let $G$ be a co-unipolar graph, with a unipolar arrangement $(\overline G, C)$ for $\overline{G}$.  Suppose that $C$ is non-empty,
    and every vertex of $C$ has neighbours in at least two
    side independent sets of $\overline C$.
    Then $G$ is $2$-clique-colourable.
  \end{claim}
  \begin{proof}
    Since $C$ is non-empty,
    there are at least two side independent sets in $\overline C$.
    Let $Q$ be a side independent set.
    Colour the vertices in $C \cup Q$ with $1$
    and the remaining vertices with $2$.
    A maximal clique contained completely in $\overline C$
    contains a vertex from each
    side independent set and therefore is not monochromatic.
    There are no maximal cliques contained completely in $C$.
    A monochromatic clique which intersects both $C$ and $\overline C$
    must contain precisely two vertices -- $c \in C$ and $q \in Q$.
    However, such a clique can always be extended with a neighbour of $c$
    from a different side independent set.
  \end{proof}
  Let us check that $G^-$ satisfies the condition of the claim wvhp.
  By Theorem~\ref{thm:concentration}
  we may assume $k \in \mathcal{I} = [n/2 - \sqrt{n}, n/2 + \sqrt{n}]$.
  Let $(G^-, C)$ be the co-unipolar arrangement induced by the generation.
  By Lemma~\ref{thm:max_block} wvhp
  the maximum size of an independent set of $\overline C$
  is at most $n/\sqrt{\ln n} = o(n)$.
  On the other hand every vertex of $C$ has degree at least $n/5$
  wvhp from Chernoff's inequality and the union bound.
  If the minimum degree of the vertices in $C$ is bigger than the
  maximum size of a side independent set $\overline C$,
  then every vertex of $C$ must see at least two side independent sets.
\end{proof}

Theorem~\ref{thm:cliques} now follows from Equation~(\ref{eq:partition_of_gs}), Lemma~\ref{thm:cliques_unipolar}, Lemma~\ref{thm:cliques_co-unipolar} and Theorem~\ref{thm:main}.
Indeed we see that the probability that $P_n$ fails to be 2-clique-colourable
is $2^{-\Omega(\log^2 n)}$.
Further, we may see from the proofs that there is an $O(n^2)$-time
algorithm that on input $P_n$ first generates a unipolar representation for the graph or its complement, and then finds a 2-clique-colouring, with failure probability $2^{-\Omega(\log^2 n)}$.


\section{Hamilton cycles in the random perfect graph~$P_n$}
\label{sec:hamiltonian}

A \emph{Hamilton} cycle in a graph is a cycle visiting every vertex exactly once.
Recall the distribution $L(n)$ in Definition~\ref{def:ln}.
In this section we prove the following theorem.
\begin{thm} \label{thm:hamilt}
  Almost all perfect graphs are Hamiltonian.
  Indeed, for $P_n \in_u \mathcal{P}_n$ and $X \sim L(n)$,
  \[
  \mathbb{P}(P_n \text{ is Hamiltonian} )
  = 1 - \frac{1}{2} \mathbb{P}(X > n/2) \pm e^{-\Omega(n)}.
  \]
\end{thm}
\noindent
It follows from this result and Theorem~\ref{thm:concentration} that
\[
\mathbb{P}(P_n \text{ is Hamiltonian} )
= 1-2^{-(\frac14 + o(1)) \log^2 n} = 1 -  o(1),
\]
which will complete the proof of~(\ref{eqn.haminsec1}).
Indeed, we may see from the proofs that there is an $O(n^3)$-time
algorithm that succeeds wvhp on input $P_n$, and either outputs a Hamilton cycle, or a stable set of size $>n/2$ (proving there is no Hamilton cycle), or fails.

To prove Theorem~\ref{thm:hamilt} we consider unipolar and co-unipolar graphs
separately,
and in order to handle co-unipolar graphs we consider random bipartite graphs.

\subsection{Hamilton cycles in random unipolar graphs}
\begin{define}
  Given a graph $G$ and two disjoint sets of vertices, $S$ and $T$,
  we say that an $S$--$T$ matching $M$ is a \emph{complete matching from $S$ to $T$}
  if the edges of $M$ cover $S$.
\end{define}
  
\begin{lemma}
  \label{thm:unipolar_deterministic}
  Suppose that $(G, C)$ is a unipolar arrangement.
  Let $T_1, T_2 \subseteq \overline C$ each contain exactly one vertex
  from each side clique and be such that $T_1 \cap T_2$ consists
  of the vertices of $\overline C$ forming side cliques of size one.
  Let $A_1$ and $A_2$ be disjoint subsets of $C$. Suppose that in $G$
  there are complete matchings $M_1$ from $T_1$ into $A_1$ and $M_2$ from $T_2$
  into $A_2$. Then $G$ is Hamiltonian.
\end{lemma}
\begin{proof}
  Fix some ordering of the side cliques of $\overline C$.
  There is a Hamilton cycle which enters each clique from an edge of $M_1$,
  visits all vertices inside and then leaves the clique from an edge of $M_2$.
  After all side cliques are visited,
  the Hamilton cycle visits any remaining vertices in $C$ and finishes at the
  initial vertex.
\end{proof}

\begin{define}
  Let $G_{n, m, \frac{1}{2}}$ denote a random $n \times m$ bipartite graph
  with the edges present independently with probability $1/2$.
\end{define}

\begin{lemma}
  \label{thm:matching_bip_graph}
  Suppose that $1 \le n \le m$.
  Then $G = (V_1, V_2, E) \sim G_{n, m, \frac{1}{2}}$ contains a complete matching
  from $V_1$ to $V_2$ with probability $1 - 2^{-m(1 + o_m(1))}$.
\end{lemma}
\begin{proof}
  We can find an isolated vertex in $V_1$ with probability at least $2^{-m}$,
  hence $2^{-m}$ is a lower bound on the probability of failing to have a
  complete matching
  (from $V_1$ to $V_2$).
  Hall's marriage theorem states that $G$ contains a complete matching
  if and only if $|N(S)| \ge |S|$ for every $S \subseteq V_1$.
Let $E_S$ be the event that $|N(S)| < |S|$. We have $\mathbb{P}(E_S) \le {m \choose m - s + 1}2^{-s(m-s+1)}$, where $s = |S|$. Thus
\begin{align*}
    \mathbb{P}(\cup_S E_S)
    &\le \sum_S \mathbb{P}(E_S)
  = 
    \sum_{|S| \le 2} \mathbb{P}(E_S) + \sum_{3 \leq |S| \leq n - 2} \mathbb{P}(E_S) + \sum_{ |S| \ge n - 1} \mathbb{P}(E_S). 
\end{align*}
Clearly
\[ \sum_{|S| \le 2} \mathbb{P}(E_S) \leq n 2^{-m} + \binom{n}{2} m 2^{-2(m-1)} \leq n \, 2^{-m +O(1)}.\]
For $3 \leq s \leq \min \{n,m-2\}$ we have $2^{-s(m-s+1)} \leq 2^{-3(m-2)}$, so
\[\sum_{3 \leq |S| \leq n - 2} \mathbb{P}(E_S)  \leq\!\!\sum_{3 \leq |S| \leq \min \{n,m-2\}} \mathbb{P}(E_S) \leq 2^{n+m} 2^{-3(m-2)} \leq 2^{-m+O(1)}. \]
If $m \geq n+2$ we are done, so assume that $m=n+\delta$ where $\delta$ is 0 or 1, and consider $s=n-1$ and $s=n$.
We have $(n-1)(m-n+2)=(m-\delta-1)(\delta+2) \geq m$ for $m \geq 3$, and $n(m-n+1) = (m-\delta)(\delta+1) \geq m$ for $m \geq 2$.  Hence
\[ \sum_{ |S| \ge n - 1} \mathbb{P}(E_S) \leq n^2 2^{-m}\]
for $m \geq 3$, which completes the proof.
\end{proof}

\begin{lemma}
  The random graph $G^+ \sim Gen^+(n)$ is Hamiltonian wvhp.
\end{lemma}
\begin{proof}
  Condition on $|k - n/2| < \sqrt{n}$, which happens with wvhp
  by Theorem~\ref{thm:concentration}.
  By Lemma~\ref{thm:chernoff_blocks_count},
  the number of side cliques is at most $n/16$ wvhp.
  Condition on this event
  and on the arrangement $(G^+, C)$ induced by the generation process.
  Partition $C$ into two near equal parts $A, B \subseteq C$ with $||A| - |B|| \le 1$,
  so that $|A| \sim |B| \sim n/4$.
  Find $T_A, T_B \subset \overline C$ as in Lemma \ref{thm:unipolar_deterministic}.
  By Lemma \ref{thm:matching_bip_graph}, we can find complete matchings between $T_A$ and $A$
  and between $T_B$ and $B$ wvhp.
  We now see that $G^+$ satisfies the conditions of Lemma \ref{thm:unipolar_deterministic}
  wvhp and therefore $G^+$ is Hamiltonian wvhp.
\end{proof}

\subsection{Hamilton cycles in random bipartite graphs}

Hamilton cycles in random bipartite graphs have been studied in
Frieze~\cite{ham_bipartite}
and in Bollob\'{a}s and Kohayakawa~\cite{bollobas_ham_cycle}.
However, we are interested in the dense case and
seek an exponentially small failure probability,
which does not appear to have been done before.
We shall show:

\begin{thm}
  \label{thm:probabilistic_hamiltonian}
  If $G \sim G_{n, n, \frac{1}{2}}$
  then $G$ is Hamiltonian with probability $1 - 2^{-(1 +o(1))n}$.
\end{thm}

This theorem will follow from the deterministic
Theorem~\ref{thm:deterministic_hamiltonian} below,
based on `Posa flips', but first we need some preparation.

\begin{define}
  Suppose $G$ is a bipartite graph.
  We define $\widetilde{\alpha}(G)$ to be the maximum integer $k$
  such that we can find $k$-sets $S_1 \subseteq V_1$ and $S_2 \subseteq V_2$
  with $E(S_1, S_2) = \emptyset$ for a bipartition $(V_1, V_2)$ of $G$.
  We call the pair $(S_1, S_2)$ \emph{a bipartite hole} of $G$.
\end{define}

\noindent
We denote the minimum degree of a graph $G$ by $\delta(G)$.

\begin{lemma}
  \label{thm:ham_connected}
Suppose $G$ is a bipartite graph such that $\delta(G) > \widetilde{\alpha}(G)$. Then $G$ is $(\delta(G) - \widetilde{\alpha}(G))$-connected.
\end{lemma}
\begin{proof}
Let $S \subseteq V(G)$ be a set of less than $\delta(G) - \widetilde{\alpha}(G)$  vertices. We must show that $G'=G[V(G) \setminus S]$ is connected. Let $(V_1, V_2)$ be a bipartition of $G$.

Let $v_1 \in V_1 \setminus S$ and $v_2 \in V_2 \setminus S$.  Then $|N(v_1) \setminus S| \ge \widetilde{\alpha}(G) + 1$ and $|N(v_2) \setminus S| \ge \widetilde{\alpha}(G) + 1$; and hence either $v_1$ and $v_2$ are adjacent, or a neighbour of $v_1$ is adjacent to a neighbour of $v_2$ in $G'$.  Thus $v_1$ and $v_2$ are connected  in $G'$ by a path of length at most $3$.
If $v_1, v_2 \in V_1 \setminus S$, then $v_1$ has a neighbour $u \in V_2 \setminus S$, and we have already seen that $v_2$ and $u$ must be connected in $G'$, so $v_1$ and $v_2$ are connected in $G'$.  The same conclusion holds if $v_1, v_2 \in V_2 \setminus S$.
\end{proof}

\begin{thm}
  \label{thm:deterministic_hamiltonian}
  An $n \times n$ bipartite graph $G$ with $n \ge 2$ is Hamiltonian if
  $\delta(G) \ge 2\widetilde{\alpha}(G) + 1$.
\end{thm}

\begin{proof}
  All paths in this proof have their terminal vertices labelled as \emph{start} and \emph{end}. Suppose $P = \{v_1, \ldots, v_l\}$ is a path with start $v_1$ and end $v_l$, and suppose $v_l$ is adjacent to $v_i$. A \emph{flip of $P$ around $v_i$} is the path $\{v_1, \ldots, v_i, v_l, \ldots, v_{i+1}\}$ over $V(P)$, starting at $v_1$ and ending at $v_{i+1}$.
  (If $i=l-1$ the flip does nothing.)

Let $P = \{v_1, \ldots, v_l\}$ be a path starting at $v_1$, ending at $v_l$, and having maximum length of any path in $G$.  For $k = 0, 1, 2$ let $W_k$ be the set of ends of paths obtained from $P$ by at most $k$ flips. Clearly $W_0 = \{v_l\}$. Since $P$ is maximal, $v_l$ is adjacent to vertices of $P$ only, hence $W_1 = \{v_{i+1} : v_iv_l \in E(G)\}$ satisfies $|W_1| \geq \delta(G)$.

Assume that $\delta(G) \ge 2\widetilde{\alpha}(G) + 1$, and let $(V_1, V_2)$ be a bipartition of $G$.
Without loss of generality suppose that $v_l \in V_1$.
Thus each set $W_k \subseteq V_1$.
To complete the proof it is sufficient to show that
every vertex of $V_2$ is adjacent to a vertex of $W_2$.
Indeed, suppose this is the case.
If $v_1 \in V_2$, then $P$ can be closed to a cycle, which must be Hamiltonian
since $G$ is connected by Lemma~\ref{thm:ham_connected} and $P$ has maximum length.
If $v_1 \in V_1$, then there is a vertex $u \in V_2 \setminus V(P)$.
But this contradicts the maximality of $P$,
because there is a path $Q$ with $V(Q) = V(P)$
ending at a neighbour of $u$ obtained from $P$ by two flips,
and $Q$ can be extended to contain $u$, yielding a longer path.

We claim that
\begin{equation} \label{claim.bipham}
|W_2| \geq n- 2 \widetilde{\alpha}(G).
\end{equation}
Assuming this, and recalling that $\delta(G) \ge 2\widetilde{\alpha}(G) + 1$, we see that every vertex of $V_2$ must be adjacent to a vertex of $W_2$, and we are done.  So it remains to establish the claim~(\ref{claim.bipham}).

  Suppose $v_{i+1} \in W_1$ is adjacent to $v_j \neq v_i$.
  Then
  \begin{align}
    v_{j+1} \in W_2 \hspace{15pt} &\mbox{ if } j < i; \label{eq:succ}\\
    v_{j-1} \in W_2 \hspace{15pt} &\mbox{ if } j > i  \label{eq:pred}.
  \end{align}
  More precisely, the path obtained from $P$ by first flipping around $v_i$ and
  then $v_j$ ends with $v_{j+1}$ or $v_{j-1}$ depending on whether $v_j$
  comes before $v_i$ in $P$.

Since $|W_1| \ge 2\widetilde{\alpha}(G) + 1$, we can find an integer $t$ such that $v_{t+1} \in W_1$ and if we write
\begin{align*}
    F_1 &= \{v_i \in V(G) : i > t+1 \} \text{ and} \\
    F_2 &= \{v_i \in V(G) : i < t+1 \},
\end{align*}
then $|F_1 \cap W_1| \ge \widetilde{\alpha}(G)$ and
$|F_2 \cap W_1| \ge \widetilde{\alpha}(G)$.
Let $M_1 = (F_1 \cap W_1) \cup \{v_{t+1}\}$ and $M_2 = (F_2 \cap W_1) \cup \{v_{t+1}\}$.
It follows from (\ref{eq:pred}) that if $v_j \in F_1$ has a neighbour in $M_2$, then its predecessor, $v_{j-1}$, is contained in $W_2$;
and from (\ref{eq:succ}) it follows that if $v_j \in F_2, v_j \neq v_t$ has a neighbour in $M_1$, then its successor, $v_{j+1}$, is contained in $W_2$.  But $v_{t+1} \in W_1$, and hence $v_{t+1} \in W_2$,
 so no special care is required for the case $v_j = v_t$.
  For $k=1, 2$ let $N_k = N[M_{3-k}] \cap F_k$ and
  $N_k^c = (F_k \cap V_2 ) \setminus N_k$
  (and hence $F_k \cap V_2 = N_k \cup N_k^c$);
  and let $N_1^-$ be the set of predecessors in $P$ of vertices in $N_1$,
  and let $N_2^+$ be the set of successors of $N_2$.
  We see that
\begin{enumerate}
  \item
    $N_1^- \cup N_2^+ \subseteq W_2 \setminus \{v_l\}$;
  \item
    $|N_1^-| = |N_i|$, $|N_2^+| = |N_2|$;
  \item
    $N_1^- \cap N_2^+ = \{v_{t+1}\}$;
  \item
    $v_l \in W_2$.
\end{enumerate}
Since $F_k \cap V_2 = N_k \cup N_k^c$ and $v_{t+1} \in V_1$, we observe that
$n = |V_2| = |N_1| + |N_1^c| + |N_2| + |N_2^c|$, and therefore
\begin{align*}
    |W_2|
    &\ge |N_1^- \cup N_2^+| + 1
    = (|N_1^-| + |N_2^+| - 1) + 1 \\
    &= |N_1| + |N_2|
    = n - (|N_1^c| + |N_2^c|).
\end{align*}
However, $E(M_{3-k}, N_k^c) = \emptyset$ (by the definition of $N_k$) and $|M_k| > \widetilde{\alpha}(G)$, so $|N_k^c| \le \widetilde{\alpha}(G)$, and hence $|W_2| \ge n - 2\widetilde{\alpha}(G)$.
Thus we have proved~(\ref{claim.bipham}), and the proof is complete.
\end{proof}

Theorem~\ref{thm:probabilistic_hamiltonian} now follows because if $G \sim G_{n, n, \frac{1}{2}}$ and $t = \sqrt n$,
then the events $\{\delta(G) < 2t+1\}$ and $\{\widetilde{\alpha}(G) > t\}$ both have probability $2^{-(1 + o(1))n}$. As a lower bound consider the event that $G$ contains an isolated vertex.
For similar results, concerning non-bipartite graphs, we refer to \cite{ham_random}. The paper \cite{ham_random} explains how Theorem~\ref{thm:probabilistic_hamiltonian} could be extended to state that we can find $n^\delta$, $0< \delta < \frac12$,
edge-disjoint Hamilton cycles with probability $1 - 2^{-(1 + o(1))n}$.

\subsection{Hamilton cycles in random co-unipolar graphs}

\begin{lemma}
  Suppose that $G^- \sim Gen^-(n)$ and $X \sim L(n)$.
  Then $G^-$ is Hamiltonian with probability $\mathbb{P}(L(n) \le n/2) - e^{-\Omega(n)}$.
\end{lemma}
\begin{proof}
  Recall that $Gen^-(n)$ is a random quadruple $(-1, E, (k, \sigma), \pi)$.
  If $k > n/2$, then $\alpha(G^-) > n/2$,
  which makes it impossible to contain a Hamilton cycle.
  We show that conditional on $k \le n/2$, $G^-$ is Hamiltonian wvhp,
  which will complete the proof.

  From Theorem~\ref{thm:concentration} and Lemma~\ref{thm:chernoff_blocks_count} we may condition on the event that $k \ge n/2 - \sqrt n$ and the event that the number of parts in $\sigma$ is at least $n / (3 \ln n)$ (since both hold wvhp).  Condition also on $(G^-, C)$ being the co-unipolar arrangement induced by the generation.  Partition $\overline C$ as $Q \cup T$, so that $|C| = |Q| = k$, $|T| = n - 2k \le 2 \sqrt n$, and $T$ contains vertices only from different parts of $\sigma$. By Corollary~\ref{thm:probabilistic_hamiltonian}, there is a Hamilton cycle in the induced subgraph $G^-[C \cup Q]$ wvhp. Condition on this event and on $H$ being a Hamiltonian cycle. If $T$ is empty we are done, so assume not.

  Now we have a Hamilton cycle $H$ in $G^-[C\cup Q]$, and a non-empty set
  $T$ of vertices inducing a clique. Let $u, v$ be vertices in $T$,
  and note that there is a $u$--$v$ Hamilton path in $G^-[T]$
  (where $u = v$ if $|T| = 1$).
  We may assume that the maximum size of a side stable set is at most
  $n/\log n$ by Lemma~\ref{thm:max_block}.
  Let $S_u, S_v$ be the side stable sets containing $u$, $v$ respectively,
  and let $R$  be the set of vertices in $Q \setminus S_u$.
  Then  $u$ is adjacent to each vertex in $R$,
  and $|R| \geq k-|S_u| \geq n/3+ |S_v|$ for $n$ sufficiently large.

  The vertices in $R$ have at least $|R| -|S_v| \geq n/3$ neighbours along $H$
  which are not in $S_v$, and if $v$ is adjacent to any of these neighbours
  then we have a Hamilton cycle in the whole graph.  But the probability that
  this fails is at most $2^{-n/3}$, since any neighbour in $\overline{C}$
  must be adjacent to $v$ and so we may assume there are none.
\end{proof}


\section{Connectivity and Chromatic index}
\label{sec:conn-chrindex}

\subsection{Connectivity}
A set of vertices in a graph $G$ is a \emph{cutset}
if removing these vertices leaves a disconnected graph.
The \emph{(vertex) connectivity} of $G$ is the minimum size of a cutset,
except that, if $G$ is complete, then by convention $\kappa(G) = |V(G)| - 1$.

We use $\delta(G)$ to denote the minimum vertex degree.
Observe that $\delta(G)$ is natural upper bound for $\kappa(G)$.
We prove the following theorem:
\begin{thm}
  \label{thm:connectivity}
  The connectivity of almost all perfect graphs equals their minimum degree.
  Indeed, wvhp $\kappa(P_n) = \delta(P_n)$, where $P_n \in_u \mathcal{P}_n$.
\end{thm}

\begin{proof}
  We first consider unipolar graphs.
  Let $(G^+, C)$ be the unipolar arrangement induced by the generation of $Gen^+(n)$.
  Then, for each fixed $\epsilon$ with $0 < \epsilon < \frac{1}{20}$,
  $G^+$ satisfies each of the following properties wvhp:
  \begin{description}
  \item{(a1)}
    $\forall u \in {\overline C}$, $d_C(u) \le (1 + \epsilon)\frac{1}{4}n$.
  \item{(a2)}
    $\forall u \in {\overline C}$, $d_{\overline C}(u) \le \epsilon n$.
  \item{(a3)}
    $\forall u, v \in {\overline C}$, $\; |(N(v) \cup N(u)) \cap C| \ge (1 - \epsilon)\frac{3}{8}n$.
  \item{(a4)}
    $|C| > (1 - \epsilon)\frac{n}{2}$.
  \end{description}
  Chernoff's inequality is sufficient to prove (a1) and (a3);
  for (a2) we can use Lemma~\ref{thm:max_block},
  and (a4) follows from the concentration theorem, Theorem~\ref{thm:concentration}.

  These conditions are sufficient to ensure $\kappa(G^+) = \delta(G^+)$.
  Note first that if ${\overline C}= \emptyset$ then $G^+$ is a clique and the result holds;
  so we may assume that ${\overline C} \neq \emptyset$.
  Suppose for contradiction that $Q \subset V=V(G^+)$ is a separator of $G^+$ with
  $|Q| < \delta(G^+)$.  Note that
  $\delta(G^+) \le (1 + 5 \epsilon)\frac{n}{4} \leq |C|$, and so $C \not \subseteq Q$.
  We claim that every vertex of $V \setminus Q$
  is connected to a vertex in $C \setminus Q$,
  which will yield our contradiction.
  This is obvious for vertices $u \in C \setminus Q$.
  If $u \in {\overline C}$, then since $|Q| < \delta(G^+)$, $u$ must have a neighbour $v$ outside $Q$.
  If $v \in C \setminus Q$ we are done, so suppose that $v \in N(u) \cap {\overline C}$.
  From (a3), $u$ and $v$ see together at least $(1 - \epsilon)\frac{3}{8}n$
  vertices of $C$, while $|Q|$ is at most $(1 + 5 \epsilon)\frac{n}{4}$,
  so there must be a vertex from $C \setminus Q$ connected to either $u$ or $v$.
  \smallskip

  The situation is similar for co-unipolar graphs.
  Let $(G^-, C)$ be the co-unipolar arrangement induced by the generation of $Gen^-(n)$.
  Then, for each fixed $\epsilon$ with $0 < \epsilon < \frac{1}{20}$,
  $G^-$ satisfies the following conditions wvhp.
  \begin{description}
  \item{(b1)}
    $\forall u \in C$, $d(u) \le (1 + \epsilon)\frac{1}{4}n$.
  \item{(b2)}
    Each side independent set in ${\overline C}$ has size at most $\epsilon n$.
  \item{(b3)}
    $|{\overline C}| > (1 - \epsilon)\frac{n}{2}$.
  \end{description}
  To prove these, we can use a Chernoff bound for (b1),
  Lemma~\ref{thm:max_block} for (b2),
  and Theorem~\ref{thm:concentration} for (b3).

  These conditions are sufficient to ensure $\kappa(G^-) = \delta(G^-)$.
  Note first that if $C = \emptyset$ then $G^-$ is a
  complete multipartite graph and the result holds; so we may assume that
  $C \neq \emptyset$.  Thus $\delta(G^-) \le (1 + \epsilon)\frac{1}{4}n$.
  Suppose that $Q \subset V=V(G^-)$ is a cutset of $G^-$ and $|Q| < \delta(G^-)$.
  Fix some vertex $w \in V \setminus Q$.
  We claim that every vertex of $V \setminus Q$ is connected to $w$.

  First consider $u \in {\overline C}$.
  Then $w$ and $u$ are not neighbours (if and) only if they are contained in the
  same side independent set $I$.
  But then there is a vertex $x \in {\overline C} \setminus Q$ in a different side independent set,
  since $|{\overline C} \setminus Q| > (1 - 3 \epsilon) \frac{n}{4} \geq |I|$;
  and the path $u,x,w$ connects $u$ to $w$.
  Now consider $u \in C$.
  Since $|Q| < \delta(G)$,
  $u$ must have a neighbour in ${\overline C}$, and we just seen that all vertices in ${\overline C}$
  are connected to~$w$.

  We have now shown that $\mathbb{P}_{Gen(n)}(\kappa(G) = \delta(G)) = 1 - e^{-\Omega(n)}$,
  and Theorem~\ref{thm:main} completes the proof.
\end{proof}

\subsection{Chromatic Index}
The \emph{chromatic index} of a graph $G$, $\chi^\prime(G)$,
is the minimum number of colours
required to colour the edges, so that no edges with the same colour share a vertex.
Vizing proved in \cite{vizing} that $\chi^\prime(G) = \Delta(G)$ or
$\chi^\prime(G) = \Delta(G) + 1$.
If $\chi^\prime(G) = \Delta(G)$, then $G$ is said to be class one,
otherwise it is said to be class two.

In this subsection we prove
\begin{thm}
  \label{thm:chromatic_idx}
  Almost all perfect graphs are class one graphs.
\end{thm}

Class one graphs have a natural connection with perfect graphs.
Denote the line graph of a graph $G$ by $L(G)$.
Then, for graphs $G$ with $\Delta(G) \geq 3$,
$\Delta(G) = \omega(L(G))$ and $\chi^\prime(G) = \chi(L(G))$;
and hence $G$ is a class one graph iff $\omega(L(G)) = \chi(L(G))$.
Vizing showed in \cite{vizing2} that each graph of class two
has at least $3$ vertices of maximum degree
(and indeed the vertices of maximum degree induce a subgraph with a cycle).
Erd{\H{o}}s and Wilson proved in \cite{chromatic_index} that almost
all graphs have a unique vertex of maximum degree, and thus are class one, see also~\cite{fjmr1988}.  We shall show that almost all perfect graphs have a unique vertex of maximum degree, and thus are class one
and can be $\Delta$-edge-coloured in $O(n^3)$ time.
(If $G$ has a unique vertex $v_0$ of maximum degree then we can $\Delta$-edge-colour $G$ in $O(nm)$-time as follows.
Pick an edge $e$ incident with~$v_0$, use for example the algorithm of~\cite{vizing_algorithm} to $\Delta(G)$-edge-colour $G\backslash e$, and then use a standard Vizing fan iteration to colour $e$.)

Let $d_j = d_j(G)$ be the $j$th largest degree in the list of all $v(G)$ degrees.
Theorem 15 in Chapter III (The degree sequence) of the book \cite{bollobas_random_graphs}
concerns the first few gaps $d_i - d_{i + 1}$ for a random graph $G_{n, p}$.
The next lemma is a much easier result that may be proved along similar lines;
we do not spell a proof here.

\begin{lemma}
  \label{coro_degree}
  Whp $B := G_{n, m, 1/2}$ with $m/n \rightarrow 1$ is such that
  \[
  d_1(B) - d_2(B) \ge \frac{n^{1/2}}{\log n}.
  \]
  where $d_1(B)$ and $d_2(B)$ are the largest and second largest degrees
  of vertices in the first colour class.
\end{lemma}

Let $(G^+, C)$ be the unipolar arrangement induced by the generation of $Gen^+(n)$,
and let $B$ denote the corresponding bipartite graph
with parts $C$ and $\overline C$.
With high probability,
$d(u) \ge (1 - \epsilon)\frac{3}{4}n$ for all $u$ in $C$,
and $d(v) \le (1 + \epsilon)\frac{1}{4}n$ for all $v$ in $\overline C$.
Hence, by Lemma~\ref{coro_degree}, whp
\[
d_1(G^+) - d_2(G^+) = d_1(B) - d_2(B) \ge \frac{n^{1/2}}{2\log n},
\]
and so whp $G^+$ has a unique vertex of maximum degree.

Let $(G^-, C)$ be the co-unipolar arrangement induced by the generation of $Gen^-(n)$,
and let $B$ denote the corresponding bipartite graph with parts
$\overline C$ and $C$ (in this order).
With high probability, $d(u) \ge (1 - \epsilon)\frac{3}{4}n$ for all $u$ in
$\overline C$, and $d(v) \le (1 + \epsilon)\frac{1}{4}n$ for all $v$ in $C$.
Also, by Lemma~\ref{thm:max_block},
with high probability $|d_{\overline C}(u)-d_{\overline C}(v)| < 10 \ln n$
for each $u, v \in \overline C$, and therefore
\begin{align*}
  d_1(G^-) - d_2(G^-)
  & \geq d_1(B) - d_2(B) - 10 \ln n \\
  &\ge (1 - \epsilon)\frac{n^{1/2}}{2\log n} - 10 \ln n,
\end{align*}
and so whp $G^-$ has a unique vertex of maximum degree.

We have now seen that whp both $G^+$ and $G^-$ have a unique vertex of maximum degree,
which implies they are class one;
Theorem~\ref{thm:chromatic_idx} now follows from Theorem~\ref{thm:main}.


\section{The limit of a sequence of random perfect graphs}
\label{sec:limits}
We show that a sequence of uniformly and independently sampled $n$-vertex
perfect graphs converges with probability one to the graphon
\[
W_P(x, y) = \frac{\mathds{1}[x \le 1/2] + \mathds{1}[y \le 1/2]}{2}
\]
and show how to use this for estimates on the subgraph densities.

\subsection{Notation}
We use the notation for left limits from \cite{lovasz_limits}.
Suppose throughout (as usual) that $P_n$ is a uniformly sampled perfect graph
on vertex set $[n]$.
Let $\lambda$ be the Lebesgue measure on $\mathbb{R}^2$.
A \emph{kernel} $W : [0, 1]^2 \to \mathbb{R}$ is an a.e. bounded symmetric
measurable function.
A \emph{graphon} $W$ is a kernel with $0 \le W \le 1$ a.s.
The \emph{cut norm}, $||.||_{\square}$, is a norm on the vector space of kernels
defined by
\[
||W||_{\square} = \sup_{S, T \subseteq [0, 1]}
\left| \int_{S \times T } W d \lambda \right |.
\]
The supremum is taken over all measurable $S, T \subseteq [0, 1]$.
Let $S_{[0, 1]}$ denote the set of invertible measure preserving maps on $[0, 1]$.
Define $W^\varphi(x, y) := W(\varphi(x), \varphi(y))$ for $\varphi \in S_{[0, 1]}$.
The \emph{cut distance}, $\delta_\square$, is defined by
\[
\delta_\square(U, W) = \inf_{\varphi \in S_{[0, 1]}}||U - W^\varphi||_\square
\] for kernels $U, W$.
For a graph $G$ we define the graphon $W_G$ by partitioning
$[0,1] = S_1 \sqcup \ldots \sqcup S_{v(G)}$, $\lambda(S_i) = 1/v(G)$,
and writing $W(x, y) = \mathds{1}_{ij \in E(G)}$ if $(x, y) \in S_i \times S_j$.
Let $\hom(F, G)$ and $\hom_\text{inj}(F, G)$
be the number of homomorphisms (edge preserving maps),
injective homomorphisms respectively, from graph $F$ to graph $G$.
Let $t(F, G) = \frac{\hom(F, G)}{n^k}$ and let
$t_\text{inj}(F, G) = \frac{\hom_\text{inj}(F, G)} {n(n - 1) \ldots (n -k + 1)}$,
where $k = v(F)$ and $n = v(G)$.
For a graph $F$ and kernel $W$ define
\[
t(F, W) = \int_{[0, 1]^{v(F)}} \prod_{ij \in E(G)} W(x_i, x_j) \prod_{i \in V(G)} d x_i.
\]
We have $t(F, G) = t(F, W_G^\varphi)$
for each pair of graphs $F$ and $G$ and each $\varphi \in S_{[0, 1]}$.
A sequence of graphs, $G_1, G_2, \ldots$,
is said to converge to a graphon $W$,
denoted by $G_n \to W$,
if $W_{G_n}$ converges to $W$ with respect to $\delta_\square$.

\subsection{Convergence results}
\begin{thm}
  \label{thm:limit_main}
  We have
  \begin{align}
    &\delta_{\square}(W_{P_n}, W_P) \le n^{-1/2} \text{ with probability }
      1 - e^{-\Omega(\sqrt{n} \log n)}, \\
    &\delta_{\square}(W_{P_n}, W_P) \le (\log n)^{-2} \text{ with probability }
     1 - e^{-\Omega(n)}. \label{eq:limit_high_prob}
  \end{align}
\end{thm}
\noindent
Before we prove Theorem~\ref{thm:limit_main} we discuss a few corollaries.
Suppose each $P_n$ is sampled independently.
\begin{cor}
  With probability one we have $P_n \to W_P$.
\end{cor}
\begin{proof}
This follows from Theorem~\ref{thm:limit_main} (either part),
a
Borel-Cantelli lemma, and Theorem 11.22 of \cite{lovasz_limits}.
\end{proof}

\begin{cor}
  \label{thm:sugbraphs}
  For a fixed graph $F$
  \[
  |t_{\text{inj}}(F, P_n) - t(F, W_P)|
  \le {v(F) \choose 2}n^{-1} + e(F)n^{-1/2}
  \]
  with probability at least
  $1 - e^{-\Omega(\sqrt{n} \log n)}$.
\end{cor}
\begin{proof}
  It is easy to see that for every graph $G$ we have
\[ |t_{\text{inj}}(F, G) - t(F, G)| \le {v(F) \choose 2} \frac1{v(G)}.\]
  Theorem 10.23 in \cite{lovasz_limits} states that
\[ |t(F, W_1) - t(F, W_2)| \le e(F) \cdot \delta_{\square}(W_1, W_2)\]
  for every two graphons $W_1$ and $W_2$.
  The rest follows from part 1 of Theorem~\ref{thm:limit_main}.
\end{proof}

\noindent
The notation in the second corollary may be misleading,
in the sense that the statement has nothing to do with graph limits.
Indeed, the value $t(F, W_P)$ can be calculated explicitly
and then used as an estimate on the number of subgraphs $F$ of $P_n$.
For example, if $F$ is the single edge $K_2$ then $t(F,W_P)=\frac12$ so
the random perfect graph $P_n$ contains $(1 + O(n^{-1/2}))\frac{n^2}{4}$ edges
with probability at least $1 - e^{-\Omega(\sqrt{n} \log n)}$;
and similarly $P_n$ contains $(1 + O(n^{-1/2}))\frac{5n^3}{96}$
triangles with probability at least $1 - e^{-\Omega(\sqrt{n} \log n)}$.
Using an inclusion-exclusion argument we can estimate the densities
of induced graphs and in particular
Corollary~\ref{thm:sugbraphs} implies that the split graphs are
the only graphs with induced subgraph density bounded away from zero.

We can be more precise for example concerning the number $e(P_n)$
of edges of $P_n$.
Note that the expected value is exactly $ \frac12 \binom{n}{2}$.
Also, whp $e(G_n^+) = \frac14 n^2 +(\frac14 +o(1))n \ln n$
and so whp $e(G_n^-) = \frac14 n^2 -(\frac14 +o(1))n \ln n$;
and hence whp $e(P_n) = \frac14 n^2 + O(n \log n)$, with a bimodal distribution.

\subsection{Proof of Theorem~\ref{thm:limit_main}}
We consider the unipolar case $G^+ \sim Gen^+(n)$ and
the co-unipolar case $G^- \sim Gen^-(n)$ separately, and start with the former.
Suppose $Gen^+(n) = (1, E, (k, \sigma), \pi)$.
Let $\varphi \in S_{[0, 1]}$ be a measure preserving map,
mapping each vertex of $G^+$ to an interval of the form
$V_i = [\frac{i-1}{n}, \frac{i}{n})$ where $i \in [n]$,
and in addition let $\varphi$ map the central clique of $W_{G^+}$
to $[0, k/n)$ and the side cliques to $[k/n, 1)$.
Pick arbitrary measurable $S, T \subseteq [0, 1]$.
To prove that $\delta_\square(G^+, W_P) \le \epsilon$ it is sufficient to show
that
\begin{align*}
  \left | \int_{S \times T} (W^\varphi_{G^+} - W_P) \right | \le \epsilon.
\end{align*}

\begin{claim}
  \label{thm:allign}
  We can find $S^\prime, T^\prime \subseteq [0, 1]$
  of the form $S^\prime = \cup_{i \in I_s}V_i$ and
  $T^\prime = \cup_{i \in I_t}V_i$,
  where $I_s, I_t \subseteq [n]$,
  such that
\[
  \left | \int_{S \times T} (W^\varphi_{G^+} - W_P) \right | \le
  \left | \int_{S^\prime \times T^\prime} (W^\varphi_{G^+} - W_P) \right |.
\]
\end{claim}
\begin{proof}
Initially set $S^\prime = S$ and $T^\prime = T$.
For $i = 1, \ldots, n$ in turn,
check if $\int_{T^\prime} (W^\varphi_{G^+} - W_P)(\frac{i-1}n, y)dy$
has the same sign as
$\int_{S^\prime \times T^\prime} (W^\varphi_{G^+} - W_P) d\lambda$,
and if so, add the entire interval $V_i$ to $S^\prime$,
or remove $V_i$ from $S^\prime$ otherwise.
These operations can only increase the absolute value of the integral.
Now repeat the same procedure for $T^\prime$.
It is clear that in the end we obtain $S^\prime$ and $T^\prime$
with the desired properties.
\end{proof}
\noindent
Every such pair $(S^\prime, T^\prime)$ is measurable.
We may assume that initially we are given $S$ and $T$ of this form.

Let $A_1 = [0, 1/2)$, $B_1 = [1/2, 1)$, $A_2 = [0, k/n)$, $B_2 = [k/n, 1)$,
$A = A_1 \cap A_2$, $B = B_1 \cap B_2$ and $I = A \cup B$.
By Theorem~\ref{thm:concentration},
wvhp both $\lambda(A)$ and $\lambda(B)$ are contained in
$\left[\frac12 - \frac{n^{-1/2}}{24}, \frac12\right]$.
Condition on this event, and now we have
\begin{align*}
  \left | \int_{S \times T}( W^\varphi_{G^+} - W_P) \right | \le
  \left | \int_{(S \times T) \cap (I \times I)} (W^\varphi_{G^+} - W_P) \right |
  + \frac{n^{-1/2}}{12}.
\end{align*}
Let
$O_1 = (S \times T) \cap (A \times A)$,
$O_2 = (S \times T) \cap (A \times B)$,
$O_3 = (S \times T) \cap (B \times A)$,
$O_4 = (S \times T) \cap (B \times B)$,
and $w_i = \left | \int_{O_i} (W^\varphi_{G^+} - W_P) \right |$.

We have $W^\varphi_{G^+} = W_P = 1$ over $O_1$ , so clearly $w_1 = 0$.
The restriction of $W^\varphi_{G^+}$ over $O_2$ is a step function corresponding
to a random bipartite graph with density $1/2$, while $W_P$ is uniformly $1/2$.
This classical example motivated the study of the cut distance,
but we give full details here.

\begin{claim}
  \label{thm:matrix}
Suppose $M$ is a random $\{-1, 1\}$-valued $m \times m$ matrix,
where each entry has expected value $0$ and is independent from the others.
Then the maximum absolute value of a rectangular sum in $M$ is at most
$\frac53 m^{3/2}$ wvhp.
\end{claim}
\begin{proof}
  Let $R \subseteq [m]$ be a subset of the rows of $M$,
  and let $Q \subseteq [m]$
  be a subset of the columns of $M$.
  Clearly $M_{Q, R} := \sum_{r \in R, q \in Q}M_{r,q} \sim 2 X - |Q||R|$,
  where $X \sim Bin(|Q||R|, 1/2)$;  so by Chernoff's inequality
\[
  \mathbb{P}(|M_{Q, R}| \ge x) =
  \mathbb{P}\left(\left|X - \frac{|Q||R|}2\right| \ge \frac x 2\right)
  \le 2e^{-\frac{x^2}{2m^2}}.
\]
Taking $x = \frac53m^{3/2}$, and the union bound over all $2^{2m}$ possible choices
for $R, Q$ completes the proof.
\end{proof}
\noindent
Setting $m=\lceil n/2 \rceil$ and rescaling in Claim~\ref{thm:matrix}
shows that $w_2$ and $w_3$ are bounded by $\frac5{12}n^{-1/2}$ wvhp.

Finally, $w_4$ is bounded by twice the number of edges in the partition part
of $G^+$ divided by $n^2$. Lemma~\ref{thm:max_block}
states that the maximum size of a part in a random partition is at most $x$ with
probability at least $1 - (1 + o(1)) e^{-x(\ln x - \ln \ln n - 2) + \ln n}$.
The maximum size of a part bounds the maximum degree.
Here we have to make a compromise between low cut distance and high probability.
To favour the former, use Lemma~\ref{thm:max_block} with
$x = \frac1{12} \sqrt n$ and deduce that
\begin{align*}
  \left | \int_{S \times T} (W^\varphi_{G^+} - W_P) \right | \le \frac 1 {\sqrt{n}}
\end{align*}
with probability $1 - e^{-\Omega(\sqrt n \log n)}$.

Now let us aim for a bound that holds wvhp.
Let $\log n \leq x \leq n/2$.
Observe that if $x$ vertices are contained in parts of $\sigma$ each of size at least $s$, then $x$ vertices are contained in at most $\lceil x/s \rceil$ parts. Thus, assuming $s \leq x$ and so $\lceil x/s \rceil \leq 2x/s$, the probability that some set of $x$ points is contained in parts of $\sigma$ each of size at least $s$ is at most
\begin{eqnarray*}
  && \binom{n}{x} \left(\left\lceil \frac{x}{s} \right\rceil \right)^{x} n
  \, \frac{B_{n-x}}{B_n}\\
  & \leq &
  \left( \frac{en}{x} \cdot \frac{2x}{s} \right)^{x} (1+o(1))
  \left(\frac{r_n}{r_{n-x}}\right)^{\frac12} n \, \exp\{-x(r_n-1+1/r_n)\}\\
  & = &
  \left( \frac{2en}{s} \right)^{x} \exp  \{-x( \ln n - (1+o(1))\log\log n)\}\\
  & = &
  \exp \{-x( \ln s - (1+o(1))\log\log n)\}.
\end{eqnarray*}
The factor $n$ in the first bound above arises since the `last part'
for the $x$-set may need to be amalgamated with some part for the
$(n\!-\!x)$-set.
Set $x= \lceil n/ (3 \log n) \rceil$ and $s=n/\log^2 n$,
to see that wvhp at most $n/ (3\log n)$
points are contained in parts of $\sigma$ each of size at least $n/\log^2 n$.
Since also wvhp no part has size greater than $n / \log n$ by
(\ref{eqn.sigmamax}),
we see that wvhp the number of edges in $G^+[\overline{S}]$ is at most
\[
(\frac12 +o(1))n \cdot \frac{n}{\log^2 n} +  \frac{n}{3\log n} \frac{n}{\log n}
\]
so wvhp $w_4 \leq \frac56 (\log n)^{-2}$.
Hence wvhp $\delta_\square(G^+, W_P) \le (\log n)^{-2}$.

Let $J$ be the graphon with $J(x,y)=1$ for each $x,y \in [0,1]$;
and note that  $J-W_{P}$ may be written as $W_{P}^{\psi}$
for some $\psi \in S_{[0, 1]}$.
Then
\[ \int_{S \times T} (W^\varphi_{G^-} - W_{P^{\psi}})
= - \int_{S \times T} (W^\varphi_{G^+} - W_{P}). \]
Hence, the bounds for $G^+$ transfer to $G^-$.
Theorem~\ref{thm:main} completes the proof.

Using~(\ref{eqn.sigmamax2}) we may see that the bound in
(\ref{eq:limit_high_prob}) is best possible, in the sense that,
if $\epsilon(n)=o(1)$,
then $\delta_{\square}(W_{P_n}, W_P) \ge \epsilon(n) (\log n)^{-2}$
holds with probability $e^{-o(n)}$.


\section{Proofs for the generation model}
\label{sec:genlemma}

We conclude this paper with the proofs of Theorems~\ref{thm:main}
and \ref{thm:concentration}, which were deferred to here.
Of course, the proofs of these theorems do not depend on any of the earlier work which used them.

\subsection{Proof of Theorem~\ref{thm:main}}
We shall prove (\ref{eq:double_counting}) and (\ref{eq:dtv_gs_gen})
and complete the proof of Theorem \ref{thm:main}.
For real numbers $p \geq 1$
we denote the $\mathcal{L}_p$ norm of $X$ in the probability space
$(\mathcal{G}_n, 2^{\mathcal{G}_n}, \mu)$ by $\|X\|_p^\mu$.
Recall that $\|X\|_p^\mu$ is non-decreasing in $p$.
(To see this, let $1 \leq p<q$ and note that $f(x)=x^{q/p}$ is convex for $x>0$: hence by Jensen's inequality,
\[ \|X\|_p^{\mu} = (\E |X|^p)^{1/p} = (f(\E|X|^p))^{1/q} \leq (\E f(|X|^p))^{1/q} = (\E|X|^q)^{1/q} = \|X\|_q^{\mu}. )  \]
Let $\nu$ and $\mu$ be discrete probability measures on (all subsets of)
${\mathcal G}_n$ such that $\nu(G)>0$ implies $\mu(G)>0$.
In this case $\nu$ is said to be \emph{absolutely continuous}
with respect to $\mu$, written $\nu \ll \mu$.
The Radon--Nikodym derivative $\frac{d\nu}{d\mu}$ is a random variable given by
$\frac{d\nu}{d\mu}(G) = \frac{\nu(G)}{\mu(G)} \mathds{1}_{\mu(G)>0}$.
We may express the total variation distance between $\nu$ and $\mu$
in terms of $\frac{d\nu}{d\mu}$ and the $\mathcal{L}_1$ norm
$\left\| \cdot \right\|_{1}^{\mu}$: we have
\[
2d_{TV}(\nu, \mu) = \sum_G | \nu(G)-\mu(G)|
= \sum_G \left|\frac{d\nu}{d\mu}(G)-1\right| \mu(G)
= \left\| \frac{d\nu}{d\mu} -1  \right\|_{1}^{\mu}.
\]
(The same holds for general, not necessarily discrete, probability measures,
provided that $\nu \ll \mu$.)
Also note that for each $G$
\[
\frac{d \mathbb{P}_{Gen(n)}}{d \mathbb{P}_{\mathcal{P}_n}}(G)
= \mathbb{P}( \rho(Gen(n)) =G) \, |{\mathcal P}_n|. \:
\]

The main result in this section is stronger than Theorem~\ref{thm:main}.
\begin{thm}
  \label{thm:strong_main}
  For each real $p \ge 1$ we have
\begin{align*}
  2d_{TV}(\mathbb{P}_{Gen(n)}, \mathbb{P}_{\mathcal{P}_n})
  = \left \| \frac{d \mathbb{P}_{Gen(n)}}{d \mathbb{P}_{\mathcal{P}_n}}
  - 1 \right\|_1^{\mathbb{P}_{\mathcal{P}_n}}
  \le \left \| \frac{d \mathbb{P}_{Gen(n)}}{d \mathbb{P}_{\mathcal{P}_n}}
  - 1 \right\|_p^{\mathbb{P}_{\mathcal{P}_n}}
  = e^{-\Theta(n)}.
\end{align*}
\end{thm}
\noindent
We have just noted the first two (in)equalities in the statement of the theorem:
the non-trivial part is the final equality.
To prove it we start from the triangle inequality
\[  \left \| \frac{d \mathbb{P}_{Gen(n)}}{d \mathbb{P}_{\mathcal{P}_n}}
  - 1 \right\|_p^{\mathbb{P}_{\mathcal{P}_n}} \leq
  \left \| \frac{d \mathbb{P}_{Gen(n)}}{d \mathbb{P}_{\mathcal{P}_n}}
- \frac{d \mathbb{P}_{\mathcal{GS}_n}}{d \mathbb{P}_{\mathcal{P}_n}}
\right\|_p^{\mathbb{P}_{\mathcal{P}_n}} +
\left \| \frac{d \mathbb{P}_{\mathcal{GS}_n}}{d \mathbb{P}_{\mathcal{P}_n}}
- 1 \right\|_p^{\mathbb{P}_{\mathcal{P}_n}}.\]
It is straightforward to show from (\ref{eq:promel_steger})
that the second term on the right is  $e^{-\Omega(n)}$.
For the first term on the right, routine manipulations yield
\[
\left \| \frac{d \mathbb{P}_{Gen(n)}}{d \mathbb{P}_{\mathcal{P}_n}}
- \frac{d \mathbb{P}_{\mathcal{GS}_n}}{d \mathbb{P}_{\mathcal{P}_n}}
\right\|_p^{\mathbb{P}_{\mathcal{P}_n}} =
\left \| \frac{d \mathbb{P}_{Gen(n)}}{d \mathbb{P}_{\mathcal{GS}_n}} - 1
\right\|_p^{\mathbb{P}_{\mathcal{GS}_n}}
{\left( \frac{|\mathcal{P}_n|}{|\mathcal{GS}_n|} \right)^{1- 1/p}}.\]
Since $\frac{|\mathcal{P}_n|}{|\mathcal{GS}_n|} = 1 + e^{-\Omega(n)}$
by (\ref{eq:promel_steger}), it
follows that it is sufficient to show that
$\left \| \frac{d \mathbb{P}_{Gen(n)}}{d \mathbb{P}_{\mathcal{GS}_n}} - 1
\right\|_p^{\mathbb{P}_{\mathcal{GS}_n}} = e^{-\Omega(n)}$
to complete the proof of Theorem \ref{thm:strong_main}.

We need a few definitions to continue.
\begin{define}
  Let $\mathcal{CGS}^+_n$ (from coloured generalised split graphs)
  be the set of all unipolar arrangements of order $n$
  and let $\mathcal{CGS}_n = \mathcal{CGS}^+_n \times \{-1, 1\}$.
\end{define}

Recall that $\mathscr{L}_n$ was introduced just before Definition~\ref{def:ln}.
We see that $|\mathcal{CGS}_n^+| = \mathscr{L}_n$
and $|\mathcal{CGS}_n| = 2\mathscr{L}_n$.
A closer look at the definitions reveals that
if $((G, C), B) \in_u \mathcal{CGS}_n$ is uniformly selected,
if we set $H_1 = G$ if $B = 1$ and $H_1 = \overline{G}$ otherwise,
and set $H_2 \sim \rho(Gen(n))$,
then $H_1$ and $H_2$ are equal in distribution.
Therefore
\[
\mathbb{P}_{Gen(n)}(G)
= \frac1{|\mathcal{CGS}_n|}
\sum_{C \subseteq [n]}
\big\{
\mathds{1}[(G, C) \in \mathcal{CGS}^+_n] +
\mathds{1}[(\overline{G}, C) \in \mathcal{CGS}^+_n]
\big\}.
\]
To simplify the notation we define
$R_+(G) = \sum_{C \subseteq [n]} \mathds{1}[(G, C) \in \mathcal{CGS}^+_n]$,
$R_-(G) = R_+(\overline{G})$ and $R(G) = R_+(G) + R_-(G)$,
so that
$\sum_G R(G) = |\mathcal{CGS}_n|$ and
\[
\mathbb{P}_{Gen(n)}(G) = \frac{R(G)}{|\mathcal{CGS}_n|}.
\]
Note that $R(G) > 0$ iff $G \in \mathcal{GS}_n$.
Now we see that (for each $G$)
\[
\frac{d \mathbb{P}_{Gen(n)}}{d \mathbb{P}_{\mathcal{GS}_n}}
= \frac{\frac{R}{|\mathcal{CGS}_n|}}{\frac1{|\mathcal{GS}_n|}}
= R \frac{|\mathcal{GS}_n|}{|\mathcal{CGS}_n|}.
\]
It is clear from the definitions that
$\|R - 1\|_1^{\mathbb{P}_{\mathcal{GS}_n}}
= \frac{|\mathcal{CGS}_n|}{|\mathcal{GS}_n|} - 1$.
We have
\begin{align*}
  \left \| \frac{d \mathbb{P}_{Gen(n)}}{d \mathbb{P}_{\mathcal{GS}_n}} - 1
  \right\|_p^{\mathbb{P}_{\mathcal{GS}_n}}
  &= \left \| R \frac{|\mathcal{GS}_n|}{|\mathcal{CGS}_n|} - 1
  \right\|_p^{\mathbb{P}_{\mathcal{GS}_n}}
  = \left \| R - \frac{|\mathcal{CGS}_n|}{|\mathcal{GS}_n|}
  \right\|_p^{\mathbb{P}_{\mathcal{GS}_n}}
  \frac{|\mathcal{GS}_n|}{|\mathcal{CGS}_n|} \\
  &\le \left \| R - 1 + 1 - \frac{|\mathcal{CGS}_n|}{|\mathcal{GS}_n|}
  \right\|_p^{\mathbb{P}_{\mathcal{GS}_n}}
  \le \left \| R - 1
  \right\|_p^{\mathbb{P}_{\mathcal{GS}_n}}
  + \left| 1 - \frac{|\mathcal{CGS}_n|}{|\mathcal{GS}_n|} \right| \\
  &= \left \| R - 1
  \right\|_p^{\mathbb{P}_{\mathcal{GS}_n}}
  + \left \| R - 1
  \right\|_1^{\mathbb{P}_{\mathcal{GS}_n}}
  \le 2 \left \| R - 1 \right\|_p^{\mathbb{P}_{\mathcal{GS}_n}}.
\end{align*}
Thus
\begin{equation} \label{eqn.normnew}
 \left \| \frac{d \mathbb{P}_{Gen(n)}}{d \mathbb{P}_{\mathcal{GS}_n}} - 1 \right\|_p^{\mathbb{P}_{\mathcal{GS}_n}}
\leq 2 \left \| R - 1 \right\|_p^{\mathbb{P}_{\mathcal{GS}_n}}.
\end{equation}

\noindent
Further simplifications lead us to
\begin{align}
  \label{eq:gs_to_unipolar}
  \|R-1\|_p^{\mathbb{P}_{\mathcal{GS}_n}}
  &= \|(R_+ - \mathds{1}_{\mathcal{GS}_n^+})
  + (R_- - \mathds{1}_{\mathcal{GS}_n^-})
  + (\mathds{1}_{\mathcal{GS}_n^+} + \mathds{1}_{\mathcal{GS}_n^-} - 1)
  \|_p^{\mathbb{P}_{\mathcal{GS}_n}} \nonumber \\
  &\le 2\|R_+ - \mathds{1}_{\mathcal{GS}_n^+}\|_p^{\mathbb{P}_{\mathcal{GS}_n}}
  + \|\mathds{1}_{\mathcal{GS}_n^+} \mathds{1}_{\mathcal{GS}_n^-}
  \|_p^{\mathbb{P}_{\mathcal{GS}_n}} \nonumber \\
  &= 2 \left \|R_+ - 1\right \|_p^{\mathbb{P}_{\mathcal{GS}_n^+}}
 \left( \frac{|\mathcal{GS}_n^+|}{|\mathcal{GS}_n|} \right)^{1/p}
  + \|\mathds{1}_{\mathcal{GS}_n^+} \mathds{1}_{\mathcal{GS}_n^-}
  \|_p^{\mathbb{P}_{\mathcal{GS}_n}} \nonumber \\
  &\le 2 \left \|R_+ - 1 \right \|_p^{\mathbb{P}_{\mathcal{GS}_n^+}}
  + \|\mathds{1}_{\mathcal{GS}_n^+} \mathds{1}_{\mathcal{GS}_n^-}
  \|_1^{\mathbb{P}_{\mathcal{GS}_n}}.
\end{align}

The only combinatorial argument we need to complete this proof
is phrased as a lemma with a proof in the next section:
\begin{lemma}
  \label{thm:decomposition}
  For every integral $j \ge 1$ we have
\[
  \|R_+^j\|_1^{\mathbb{P}_{\mathcal{GS}_n^+}}
  \le \|R_+\|_1^{\mathbb{P}_{\mathcal{GS}_n^+}}(1 + 2^{-n/2 + o(n)}).
\]
\end{lemma}
\noindent
From now on until equation~(\ref{eq:gs_to_unipolar_1}), all norms are for $\mu= \mathbb{P}_{\mathcal{GS}_n^+}$,
and we drop the superscript $\mathbb{P}_{\mathcal{GS}_n^+}$ for legibility. By combining the lemma for $j=2$ with Jensen's inequality we get
\[
\|R_+\|^2_1 \le \|R_+^2\|_1 \le \|R_+\|_1(1 + e^{-\Omega(n)}),
\]
and so $\|R_+\|_1 \le 1 + e^{-\Omega(n)}$.
Hence, by the last lemma, we have $\|R_+^j\|_1 \le 1 + e^{-\Omega(n)}$.
Now $1 \le R_+^j \le R_+^{j+1}$ a.s.
(that is, for all $G \in \mathcal{GS}_n^+$), and therefore
\[
\|R_+^{j+1} - R_+^j\|_1 = \|R_+^{j+1}\|_1 - \|R_+^j\|_1 = e^{-\Omega(n)}.
\]
Finally, we see that for each $j \ge 1$
\begin{align*}
  \|(R_+ - 1)^{j+1}\|_1
  &= \left\|(R_+ - 1)\sum_{i=0}^j {j \choose i} R_+^i(-1)^{j-i}\right\|_1
  \le \sum_{i=0}^j {j \choose i} \left\|(R_+ - 1) R_+^i\right\|_1 \nonumber \\
  & \le 2^j \left\|(R_+ - 1) R_+^j\right\|_1 \le e^{-\Omega(n)}.
\end{align*}
\noindent
Therefore
\begin{equation}  \label{eq:gs_to_unipolar_1}
\|R_+ - 1\|_p^{\mathbb{P}_{\mathcal{GS}_n^+}} = e^{-\Omega(n)}
\end{equation}
for each real $p \ge 1$.

By Lemma~\ref{thm:intersect} we have
$\mathbb{P}_{\mathcal{GS}^+_n}(\alpha(G) \ge \omega(G)) = e^{-\Omega(n)}$ and
$\mathbb{P}_{\mathcal{GS}^-_n}(\alpha(G) \le \omega(G)) = e^{-\Omega(n)}$.
Partition the set of graphs which are both unipolar and co-unipolar into
$U_1 = \{ G \in \mathcal{GS}^+_n \cap \mathcal{GS}^-_n : \alpha(G) > \omega(G) \}$ and
$U_2 = \{ G \in \mathcal{GS}^+_n \cap \mathcal{GS}^-_n : \alpha(G) \le \omega(G) \}$.
We see from Lemma~\ref{thm:intersect} that $U_1$ consists of rare unipolar graphs and $U_2$ consists of rare co-unipolar graphs; or more formally,
$|U_1| = e^{-\Omega(n)}|\mathcal{GS}_n^+|$ and $|U_2| = e^{-\Omega(n)} |\mathcal{GS}_n^-|$.
Hence
\begin{align}
  \label{eq:gs_to_unipolar_2}
  \|\mathds{1}_{\mathcal{GS}_n^+} \mathds{1}_{\mathcal{GS}_n^-}
  \|_1^{\mathbb{P}_{\mathcal{GS}_n}}
  = |\mathcal{GS}^+_n \cap \mathcal{GS}^-_n| / |\mathcal{GS}_n|  = e^{-\Omega(n)}.
\end{align}

\noindent
Now  (\ref{eq:gs_to_unipolar}), (\ref{eq:gs_to_unipolar_1}) and (\ref{eq:gs_to_unipolar_2}) imply that
\begin{equation}\label{eqn.Rnorm}
 \|R-1\|_p^{\mathbb{P}_{\mathcal{GS}_n}} = e^{-\Omega(n)}.
\end{equation} 

We now have all the pieces to complete the story, assuming Lemma~\ref{thm:decomposition}.
By~(\ref{eq:gs_to_unipolar_1}), 
$ {\mathcal L}_n / |\mathcal{GS}_n^+| \: - 1 = \| R^+-1\|_1^{\mathbb{P}_{\mathcal{GS}_n^+}} =  e^{-\Omega(n)}$,
so using also~(\ref{eq:gs_to_unipolar_2})
\[ |\mathcal{GS}_n| = 2 |\mathcal{GS}_n^+| - |\mathcal{GS}^+_n \cap \mathcal{GS}^-_n| = 2 {\mathcal L}_n (1-e^{-\Omega(n)}),
\]
and we have proved~(\ref{eq:double_counting}).
Also, by~(\ref{eqn.normnew}) and~(\ref{eqn.Rnorm}),
\[ 
\left \| \frac{d \mathbb{P}_{Gen(n)}}{d \mathbb{P}_{\mathcal{GS}_n}} - 1 \right\|_p^{\mathbb{P}_{\mathcal{GS}_n}} =  e^{-\Omega(n)},\]
which yields~(\ref{eq:dtv_gs_gen}), and completes the proof of Theorem \ref{thm:strong_main}.
Note finally that by the last equation and~(\ref{eq:gs_to_unipolar_2}), for $G \sim Gen(n)$
\[ \mathbb{P}( G \in \mathcal{GS}_n^+ \cap \mathcal{GS}_n^-)
\leq (1+ e^{-\Omega(n)}) \,   |\mathcal{GS}^+_n \cap \mathcal{GS}^-_n| / |\mathcal{GS}_n| =  e^{-\Omega(n)}, \]
proving~(\ref{eqn.smallcap}).

\subsection{Proving Lemma~\ref{thm:decomposition}}
Fix $j \ge 1$.
Suppose throughout this section that
$G \in \mathcal{GS}^+_n$, $C_1, \ldots, C_j \subseteq V(G)$ and $(G, C_i)$
is a unipolar arrangement of order $n$ for each $i \in [j]$.
(For this given $G$, the number of choices for $C_1,\ldots,C_j$ is $R^+(G)^j$,
which we want to upper bound.)
Further, let $C = \cap_{i=1}^j C_i$,
$\widetilde{C} = \cap_{i=1}^j \overline{C_i}$,
and let $l = n - |C| - |\widetilde{C}|$.
The following claim is easy to verify.

\begin{claim}
  The pair $(G[C \cup \widetilde{C}], C)$ is a unipolar arrangement of
  order $|C \cup \widetilde{C}|$.
\end{claim}
\noindent
Now let $C^\prime_i = C_i \setminus C$ and
$Q = V(G) \setminus (C \cup \widetilde{C})$,
so that $Q = \cup_{i=1}^j C^\prime_i$.

\begin{lemma}
Let $i \in [j]$. Then the pair $(G[Q], C^\prime_i)$ is a unipolar arrangement of order $l$ with at most $j-1$ side cliques.
\end{lemma}
\begin{proof}
  Without loss of generality suppose that $i = 1$.
  It is easy to see that $G[Q]$ is a generalised split graph
  with a central clique $C^\prime_1$,
  since $C^\prime_1 \subseteq C_1$,
  $Q \setminus C^\prime_1 \subseteq \overline{C_1}$
  and $(G, C_1)$ is a unipolar arrangement.
  Further, $Q = \cup_{i=1}^j C^\prime_i$,
  so every vertex in $Q \setminus C^\prime_1$ is covered by a clique
  $C^\prime_i$, where $2 \le i \le j$.
  Hence $G[Q \setminus C^\prime_1]$ is a disjoint union of cliques,
  covered by $j-1$ cliques,
  so the unipolar arrangement $(G[Q], C^\prime_1)$ contains at most $j-1$ side cliques.
\end{proof}
\noindent
Since each $C^\prime_i$ is complete to $C$ and $Q = \cup_{i=1}^j C^\prime_i$,
it follows that $Q$ is complete to $C$.

We now focus on $C^\prime_1$.
Observe that, for each vertex $v \in C^{\prime}_1$,
since $v \not\in C$ we can find $C^{\prime}_i$
such that $v \not\in C^{\prime}_i$.
Thus $C^\prime_1$ can be expressed as a disjoint union
$C^{\prime\prime}_2 \cup \cdots \cup C^{\prime\prime}_j$
where $C^{\prime\prime}_i \cap C^{\prime}_i = \emptyset$
for each $i=2,\ldots,j$.
Let $\mathcal C$ denote the family of sets
$\{C^\prime_i \setminus C^\prime_1\}_{i=2}^j$
$\cup \{C^{\prime\prime}_i\}_{i=2}^j$.
Note that the union of these sets is $Q$.

\begin{lemma}
For each $S \in {\mathcal C}$, either $S$ is complete to a side clique of $(G[C \cup \widetilde{C}], C)$ and non-adjacent to all other side cliques, or $S$ is non-adjacent to $\widetilde{C}$.
\end{lemma}
\begin{proof}
In the unipolar arrangement $(G, C_1)$, each $C^\prime_i \setminus C^\prime_1$ is contained in a side clique, so it is either complete to a side clique of $(G[C \cup \widetilde C], C)$, and non-adjacent to the others, or non-adjacent to any vertex of $\widetilde C$. Similarly, in the unipolar arrangement $(G, C_i)$, $C^{\prime\prime}_i$ is contained in a side clique, hence it has the same property.
\end{proof}

Let $f$ be the function mapping each set $S$ in $\mathcal C$ to the side clique $S^\prime$ of $(G[C \cup \widetilde C], C)$
such that $S$ is complete to $S^\prime$,
or to $\emptyset$ if no vertex in $S$ is adjacent to $\widetilde C$.

\begin{lemma}
The edge set $E(Q, \widetilde C)$ can be reconstructed from $f$.
\end{lemma}
\begin{proof}
For each $v \in Q$ either $v \in C^{\prime}_1$ and then $v$ lies in some set $S \in \{C^{\prime\prime}_i\}_{i=2}^j$, or $v$ lies in some set $S \in \{C^\prime_i \setminus C^\prime_1\}_{i=2}^j$.  In either case, $S$ is in $\mathcal C$ and the set of neighbours of $v$ in $\widetilde C$ is precisely the set $f(S)$.
\end{proof}

The edges of $G$ can be reconstructed
from $G[C \cup \widetilde C]$ -- a unipolar graph,
$G[Q]$ -- a unipolar graph with at most $k-1$ side cliques
and $E(Q, C \cup \widetilde C)$,
which in turn can be reconstructed from $f$,
since $Q$ is complete to $C$.
It will turn out that the number of choices for
$G[C \cup \widetilde C]$ dominates the rest,
and this number is maximised whenever $Q = \emptyset$.
However, the vertex labels are a nuisance,
and additional information is required to correctly recover $G$.
The precise statement is given below.

\begin{lemma}
The $j+1$ tuple $(G, C_1, \ldots, C_j)$ can be uniquely reconstructed from:
\begin{enumerate}
  \item
    the pair $(G[C \cup \widetilde C], C)$ seen as a unipolar
    $(n - l)$-arrangement over the vertex set $[n - l]$
    so that the order of the labels of the vertices is preserved;
  \item
    the pair $(G[Q], C^\prime_1)$ seen as a unipolar $l$-arrangement
    over the vertex set $[l]$ with at most $j-1$ side cliques
    so that the order of the labels of the vertices is preserved;
  \item
    the sets $C^\prime_2, \ldots, C^\prime_j \subseteq [l]$;
  \item
    the $l$-subset of $V(G)$ specifying the original labels of $Q$;
  \item
    the function $f$.
  \end{enumerate}
\end{lemma}
\noindent
Observe that there are at most $j^l 2^{\frac{l^2}{4}}$ choices for the pair
$(G[Q], C^\prime_1)$ above.

Recall that $\mathscr{L}_n = |\mathcal{CGS}^+_n|$.
Let $T_n = 2^{\frac{n^2}{4} + \frac{n}{2} \log n -
  \frac{n}{2} \log \ln n - \frac{n}{2} \log \frac{e}{2}}$ for $n \geq 3$,
and let $T_n=1$ for $n =0,1,2$.
Pr{\"o}mel and Steger show in \cite{promelsteger} that
$\mathscr{L}_n \le T_n \times 2^{O\left(\frac{n}{\ln \ln (n + 3)}\right)}$,
hence for some $c$ we have
$|\mathcal{CGS}^+_n| \le T_n \times 2^{\frac {c n}{\ln \ln (n+3)}}$.
It follows that
\[
U_{n, l}
:= \left (T_{n-l} \times 2^{\frac {cn}{\ln \ln (n+3)}} \right)
\times \left(j^l 2^{\frac{l^2}{4}}\right)
\times \left(2^{(j-1)l}\right)
\times {n \choose l}
\times \left((n+1)^{2(j-1)}\right)
\]
is an upper bound on the number of choices for $(G, C_1, \ldots, C_j)$
when $|Q| = l$, and $U_n = \sum_{l=1}^n U_{n, l}$
is an upper bound on the number of choices for $(G, C_1, \ldots, C_j)$
when $Q \neq \emptyset$.

\begin{lemma}
The number $U_{n, l}$ is maximised  subject to $1 \le l  \le n$ when $l = 1$.
\end{lemma}
\begin{proof}
  We may prove this by thinking of $l$ as a continuous variable and taking
  the derivative of $U_{n, l}$ with respect to $l$.
  The analysis is straightforward but unpleasant:
  we leave the details to the reader.
\end{proof}

\begin{cor}
  We have
\[
  U_n = |\mathcal{CGS}^+_n| \, 2^{-n/2 + o(n)}.
\]
\end{cor}
\begin{proof}
  Since $T_{n-1} = T_n \, 2^{-n/2 + o(n)}$,
  we see that $\frac{U_{n, 1}}{|\mathcal{CGS}^+_n|} = 2^{-n/2 + o(n)}$,
  and hence
  \[
  |\mathcal{CGS}^+_n| \, 2^{-n/2 + o(n)}
  = U_{n, 1} \le U_n \le nU_{n, 1}
  = |\mathcal{CGS}^+_n| \, 2^{-n/2 + o(n)}. \qedhere
  \]
\end{proof}

Finally, the number of choices for the $j+1$ tuple $(G, C_1, \ldots, C_j)$ when
$Q = \emptyset$ is clearly $|\mathcal{CGS}^+_n|$, so we have
\begin{align*}
  \|R^j_+\|^{\mathbb{P}_{\mathcal{GS}^+_n}}_1
  &= \sum_{G \in \mathcal{GS}^+_n} R_+^j (G) \frac1{|\mathcal{GS}^+_n|}
  \le \left(U_n + |\mathcal{CGS}^+_n|\right) \frac1{|\mathcal{GS}^+_n|}\\
  &= |\mathcal{CGS}^+_n|(1 + 2^{-n/2 + o(n)}) \frac1{|\mathcal{GS}^+_n|}
  = \|R_+\|^{\mathbb{P}_{\mathcal{GS}^+_n}}_1(1 + 2^{-n/2 + o(n)}).
\end{align*}
This completes the proof of Lemma~\ref{thm:decomposition},
Theorem~\ref{thm:strong_main} and Theorem~\ref{thm:main}.

\subsection{Concentration: Proof of Theorem~\ref{thm:concentration}}
\label{sec:concentration}
Let $\mu = (n - \log n + \log \ln n) / 2$ as in the statement
of the theorem (for $n \geq 3$),
let $\hat \mu = \lceil \mu \rceil$
and let $\check \mu = \hat \mu - 1$.
As before, let $\ell_{n, k} = {n \choose k}2^{k(n-k)}B_{n-k}$.
We fix $n$ and drop it in the subscript for legibility, i.e.
write $\ell_{k}$ instead of $\ell_{n, k}$.

\begin{lemma}
Using the notation above, there is an $n_0$ such that
$\ell_{k}/\ell_{\hat \mu} \le n^{-n-1}$ as long as $|k- \hat \mu| \ge n^{2/3}$
for $n \ge n_0$.
\end{lemma}
\begin{proof}
  Let $x = k - \hat \mu$, we have
  \begin{align*}
    \frac{\ell_k}{\ell_{\hat \mu}}
    = \frac{{n \choose k}2^{k(n-k)}B_{n-k}}
    {{n \choose \hat \mu}2^{\hat \mu(n - \hat \mu)}B_{n-\hat \mu}}
    = \frac{(n - \hat \mu)!}{(n - \hat \mu - x)!}
    \left( \frac{(\hat \mu + x)!}{\hat \mu!} \right)^{-1}
    2^{x(n - 2\hat \mu)}
    2^{-x^2}
    \frac{B_{n - \hat \mu - x}}{B_{n - \hat \mu}}.
  \end{align*}
  We can bound each of the terms
  $\frac{(n - \hat \mu)!}{(n - \hat \mu - x)!}$,
  $\left( \frac{(\hat \mu + x)!}{\hat \mu!} \right)^{-1}$,
  $2^{x(n - 2\hat \mu)}$ and
  $\frac{B_{n - \hat \mu - x}}{B_{n - \hat \mu}}$ with $n^{|x|}$,
  so we get
\begin{align*}
    \frac{\ell_{k}}{\ell_{\hat \mu}}
    <   n^{|x|} n^{|x|} 2^{-|x|^2} n^{|x|}
    =    2^{3|x| \log n - |x|^2}
    \le  2^{3n^{2/3} \log n - n ^{4/3}}
    \le n^{-n-1}
\end{align*}
  for large $n$.
\end{proof}

The bound is rather crude, and it can be improved,
but we do not need a better bound for our purposes.
From now we focus on $k = \hat \mu + x$ with $|x| \le n^{2/3}$.
This implies $k \sim n/2$, which allows us to use asymptotics
for $k$ as $n$ grows,
for instance $(n - k) / (k + 1) \rightarrow 1$ and more importantly
\begin{align*}
  \frac{B_{n-k-1}}{B_{n-k}}
  &= (1 + o_n(1)) \frac{\ln (n-k)}{n-k}
  = (1 + o_n(1))\frac{2 \ln n}{n},
\end{align*}
where the error term is uniform over $k$, provided that $|x| \le n^{2/3}$.

\begin{lemma}
  \label{thm:l_bnds}
  There is an $n_0$ such that the following holds for each $n \ge n_0$.
  Suppose $0 \le x \le n^{2/3}$ is integral and
  $r \in \{\frac{\ell_{\check \mu - x}}{ \ell_{\check \mu}}
  ,\frac{\ell_{\hat \mu + x}}{ \ell_{\hat \mu}}\}$.
  Then $2^{-x^2 - 2x} \le r \le 2^{-x^2 + 2x}$.
\end{lemma}
\begin{proof}
  Define the continuous, strictly decreasing function
  $d(y) = 2^{n -2y}\frac{\ln n}{n}$, so that for $k \sim n/2$
  \begin{align*}
    \frac{\ell_{k+1}}{\ell_k}
    =\frac{n - k}{k + 1}
    2^{n -2k -1}
    (1 + o_n(1))\frac{2\ln n}{n}
    = (1 + o_n(1)) d(k).
  \end{align*}
  Observe that
  $d(k) = 1 \iff 2k = n + \log \ln n - \log n \iff k = \mu$.
  We have
  $d(\mu + x) =  d(\mu) \cdot 2^{-2x} = 2^{-2x}$,
  and
  \begin{align*}
    \ell_{\hat \mu + x}
    & = \ell_{\hat \mu}
    \prod_{i=0}^{x - 1}
    \frac{\ell_{\hat \mu + i + 1}}
         {\ell_{\hat \mu + i}}
         = \ell_{\hat \mu}
         \prod_{i=0}^{x - 1} (1 + o_n(1)) d(\hat \mu + i).
  \end{align*}
  For each $i$ we use the bounds
  \[
  2^{-2i - 2} = d(\mu + 1 + i)
  \le d(\hat \mu + i)
  \le d(\mu + i) = 2^{-2i},
  \]
  while the error terms, $1 + o_n(1)$, we bound uniformly by $\frac12$ and $2$,
  provided $n \ge n_0$.
  By multiplying the terms we get $-2{x \choose 2} - 3x$
  and $-2{x \choose 2} + x$ in the exponent for the lower and upper bounds
  respectively,
  which matches the statement of the lemma for the case
  $r=\ell_{\hat \mu + x}/\ell_{\hat \mu}$.
  The second case is analogous.
\end{proof}

\begin{claim}
  \label{thm:integral}
  For $a > 1$ we have
  $\sum_{x=a}^\infty 2^{-x^2} < \frac{2^{-(a-1)^2}}{2(a-1) \ln 2}$.
\end{claim}
\begin{proof}
  Since $-x^2$ is decreasing, we have
  \begin{align*}
    \sum_{x=a+1}^\infty 2^{-x^2}
    &< \int_a^\infty 2^{-x^2} dx
    \le \frac 1a \int_a^\infty x2^{-x^2} dx\\
    &= \frac 1{-2a \ln 2} \left( 2^{-x^2} \bigg\vert_{x=a}^{x=\infty} \right)
    = \frac{2^{-a^2}}{2a \ln 2}.
    \qedhere
  \end{align*}
\end{proof}

In the claim below and its proof,
for simplicity we have written $n^{2/3}$
rather than $\lfloor n^{2/3} \rfloor$.
\begin{claim}
  For each integer $k \ge 2$ we have
  \[
  \sum_{i = k}^{n^{2/3}}\ell_{\check \mu - i}
  \le \ell_{\check \mu} 2^{-(k-1)^2} (2+ \frac{1}{(k-1) \ln 2})
  \hspace{10pt} \text{and} \hspace{10pt}
  \sum_{i = k}^{n^{2/3}}\ell_{\hat \mu + i}
  \le \ell_{\hat \mu} 2^{-(k-1)^2} (2+ \frac{1}{(k-1) \ln 2}).
  \]
\end{claim}

\begin{proof}
  Using Lemma~\ref{thm:l_bnds} we see that
  \begin{align*}
  \sum_{i = k}^{n^{2/3}}\ell_{\check \mu - i}
  &\le \ell_{\check \mu} \sum_{i = k}^{n^{2/3}} 2^{-i^2 + 2i}
  = 2\ell_{\check \mu} \sum_{i = k}^{n^{2/3}} 2^{-(i - 1)^2}
  \le 2\ell_{\check \mu}
  \left (2^{-(k - 1)^2} +  \sum_{i = k}^{\infty} 2^{-i^2} \right) \\
  &=2\ell_{\check \mu}
  \left (2^{-(k - 1)^2} +  \frac{2^{-(k-1)^2}}{2(k-1) \ln 2} \right).
  \end{align*}
  The last inequality follows from Claim~\ref{thm:integral}.
  The second part of the statement is similar.
\end{proof}

Now we put everything together to complete the proof of
Theorem~\ref{thm:concentration}.
Suppose $n \geq 3$, $X \sim L(n)$ and $x > 2$,
then (noting the trivial bound
${\mathcal L}_n \geq \ell_{\hat \mu} + \ell_{\check \mu}$)
\begin{align*}
  \mathbb{P}(|X - \mu| \ge x)
  \le\hspace{5pt}
  &\mathbb{P}(\check \mu - n^{2/3} \le X \le \check \mu - x + 1) \\
  + &\mathbb{P}(\hat \mu + x - 1 \le X \le \hat \mu + n^{2/3}) \\
  + &\mathbb{P}(|X - \mu| \ge n^{2/3}) \\
  \le\hspace{5pt}
  & \frac{\ell_{\hat \mu}}{\ell_{\hat \mu} + \ell_{\check \mu}}
  \cdot 2^{-(x-2)^2} \cdot 4
  + \frac{\ell_{\check \mu}}{\ell_{\hat \mu} + \ell_{\check \mu}}
  \cdot 2^{-(x-2)^2} \cdot 4  + n^{-n} \\
  = \hspace{5pt}
  &2^{-(x-2)^2 + 2} +n^{-n}, \text{ and} \\
  \mathbb{P}(|X - \mu| \ge x)
  \ge\hspace{5pt}
  & \mathbb{P}(X = \hat \mu + x)
  + \mathbb{P}(X = \check \mu - x) \\
  \ge\hspace{5pt}
  & \frac{2^{-x^2-2x} (\ell_{\hat \mu} + \ell_{\check \mu})}
            {(3+\frac1{2\ln2})(\ell_{\hat \mu} + \ell_{\check \mu})
              + n^{-n}\hat \mu} \\
  \ge\hspace{5pt}
  & 2^{-x^2-2x - 2}.
\end{align*}

\section{Concluding remarks}
\label{sec:concl}

In this paper we presented a generation model for perfect graphs
which yields such graphs almost uniformly, with error $e^{-\Theta(n)}$
(in terms of total variation distance);
and we used this approach to investigate several questions about such graphs.

We aimed to investigate the most natural questions about random perfect graphs,
but of course there are further natural open problems.
For example, concerning induced subgraphs $H$,
we found the limiting probability that $P_n$ has an induced subgraph $H$,
but we did not discuss the distribution of the number of such induced subgraphs.

Our methods do not let us approach the following question about automorphisms.
It is well known that almost all graphs $G_n$ on vertex set $[n]$
have no non-trivial automorphisms,
and easy to see that in fact the probability that $G_n$ has a non-trivial automorphism
is $2^{-(1+o(1))n}$.
It is also not hard to see that,
if $R_n \sim Gen(n)$,
then the probability that $R_n$ has a non-trivial automorphism is $2^{-(\frac12+o(1))n}$.
It is natural to expect that the probability for $P_n$ is similar to that for $R_n$,
but we do not know: we would need a finer result than~(\ref{eq:promel_steger}).
Similarly, if $H$ is a fixed graph in $\mathcal{GS}^+ \cap \mathcal{GS}^-$
and $R_n \sim Gen(n)$,
then by Lemma~\ref{thm:elaborate_subgraph} the probability that $R_n$ fails to have
an induced copy of $H$ is $e^{-\Omega(n \log n)}$.
Does such a result hold for $P_n$?

A long standing open problem is to describe the structure of a random perfect graph $G$ with a prescribed number of edges, say $e(G) \sim v(G) \log v(G)$.
This problem has been discussed before~\cite{promelsteger},
and even directly approached~\cite{taraz},
but to the best of our knowledge a complete answer is a distant possibility.

Let $C_5$-{free} denote the class of graphs with no induced subgraph $C_5$.
Then ${\mathcal GS} \subset {\mathcal P} \subset C_5{\rm -free}$,
and the key result~(\ref{eq:pre_promel_steger})
from Theorem~2.4 of Pr{\"o}mel and Steger \cite{promelsteger}
in fact extends as follows:
\begin{equation}
  \label{eq:pre_promel_steger2}
  |\mathcal{GS}_n| = (1 - e^{-\Omega(n)})|(C_5{\rm -free})_n|.
\end{equation}
Thus our results refer also to random $C_5$-free graphs.
For related general results see~\cite{h_free_1, h_free_2} and references there.

\bibliographystyle{alpha}
\bibliography{perfect}

\newcommand{\etalchar}[1]{$^{#1}$}
\begin{thebibliography}{ABBM11}

\bibitem[ABBM11]{h_free_1}
N~Alon, J~Balogh, B~Bollob{\'a}s, and R~Morris.
\newblock The structure of almost all graphs in a hereditary property.
\newblock {\em Journal of Combinatorial Theory, Series B}, 101(2):85--110,
  2011.

\bibitem[BB11]{h_free_2}
J~Balogh and J~Butterfield.
\newblock Excluding induced subgraphs: critical graphs.
\newblock {\em Random Structures \& Algorithms}, 38(1-2):100--120, 2011.

\bibitem[BGG{\etalchar{+}}04]{cliques}
G~Bacs{\'o}, S~Gravier, A~Gy{\'a}rf{\'a}s, M~Preissmann, and A~Sebo.
\newblock Coloring the maximal cliques of graphs.
\newblock {\em SIAM Journal on Discrete Mathematics}, 17(3):361--376, 2004.

\bibitem[BK91]{bollobas_ham_cycle}
B~Bollob{\'a}s and Y~Kohayakawa.
\newblock The hitting time of hamilton cycles in random bipartite graphs.
\newblock {\em Graph Theory, Combinatorics, Algorithms, and Applications
  (Philadelphia)(Y. Alavi, F. Chung, R. Graham, and D. Hsu, eds.), SIAM}, pages
  26--41, 1991.

\bibitem[Bol01]{bollobas_random_graphs}
B{\'e}la Bollob{\'a}s.
\newblock {\em Random Graphs}, volume~73.
\newblock Cambridge university press, 2001.

\bibitem[BTW12]{taraz}
J~B{\"o}ttcher, A~Taraz, and A~W{\"u}rfl.
\newblock Perfect graphs of fixed density: Counting and homogeneous sets.
\newblock {\em Combinatorics Probability and Computing}, 21(5):661, 2012.

\bibitem[CPTT16]{cliques_disproff}
P~Charbit, I~Penev, S~Thomass{\'e}, and N~Trotignon.
\newblock Perfect graphs of arbitrarily large clique-chromatic number.
\newblock {\em Journal of Combinatorial Theory, Series B}, 116:456--464, 2016.

\bibitem[DB70]{asymptotic_methods}
Nicolaas~Govert De~Bruijn.
\newblock {\em Asymptotic methods in analysis}, volume~4.
\newblock North-Holland Publishing, 1970.

\bibitem[DSSW91]{cliques_conjecture}
D~Duffus, B~Sands, N~Sauer, and R~Woodrow.
\newblock Two-colouring all two-element maximal antichains.
\newblock {\em Journal of Combinatorial Theory, Series A}, 57(1):109--116,
  1991.

\bibitem[EW77]{chromatic_index}
P~Erd{\H{o}}s and R~Wilson.
\newblock On the chromatic index of almost all graphs.
\newblock {\em Journal of Combinatorial Theory, Series B}, 23(2):255--257,
  1977.

\bibitem[FJMR88]{fjmr1988}
A~Frieze, B~Jackson, C~McDiarmid, and B~Reed.
\newblock Edge-colouring random graphs.
\newblock {\em Journal of Combinatorial Theory, Series B}, 45(2):135--149,
  1988.

\bibitem[Fri85]{ham_bipartite}
Alan Frieze.
\newblock Limit distribution for the existence of hamiltonian cycles in random
  bipartite graphs.
\newblock {\em European Journal of Combinatorics}, 6:327 -- 334, 1985.

\bibitem[Har66]{stirling_is_normal}
L~Harper.
\newblock {Stirling behavior is asymptotically normal}.
\newblock {\em Ann. Math. Stat.}, 38:410--414, 1966.

\bibitem[Lov72a]{weak_perfect1}
L{\'a}szl{\'o} Lov{\'a}sz.
\newblock A characterization of perfect graphs.
\newblock {\em Journal of Combinatorial Theory, Series B}, 13(2):95--98, 1972.

\bibitem[Lov72b]{weak_perfect2}
L{\'a}szl{\'o} Lov{\'a}sz.
\newblock Normal hypergraphs and the perfect graph conjecture.
\newblock {\em Discrete Mathematics}, 2(3):253--267, 1972.

\bibitem[Lov12]{lovasz_limits}
L{\'a}szl{\'o} Lov{\'a}sz.
\newblock {\em Large networks and graph limits}, volume~60.
\newblock American Mathematical Soc., 2012.

\bibitem[MG92]{vizing_algorithm}
Jayadev Misra and David Gries.
\newblock A constructive proof of vizing's theorem.
\newblock {\em Information Processing Letters}, 41(3):131--133, 1992.

\bibitem[MMP16]{colin_clique_colouring}
C~McDiarmid, D~Mitsche, and P~Pralat.
\newblock Clique colourings of random graphs.
\newblock 2016.

\bibitem[MY15]{recog15}
Colin McDiarmid and Nikola Yolov.
\newblock Recognition of unipolar and generalised split graphs.
\newblock {\em Algorithms}, 8(1):46--59, 2015.

\bibitem[MY17]{ham_random}
Colin McDiarmid and Nikola Yolov.
\newblock Hamilton cycles, minimum degree, and bipartite holes.
\newblock {\em Journal of Graph Theory}, 2017.

\bibitem[Pit97]{pittel_new}
Boris Pittel.
\newblock Random set partitions: Asymptotics of subset counts.
\newblock {\em Journal of Combinatorial Theory, Series A}, 79(2):326 -- 359,
  1997.

\bibitem[PS92]{promelsteger}
H~Pr{\"o}mel and A~Steger.
\newblock Almost all {Berge} graphs are perfect.
\newblock {\em Combinatorics, Probability and Computing}, 1(01):53--79, 1992.

\bibitem[RAR01]{reed_perfect_book}
J~Ram{\'\i}rez-Alfons{\'\i}n and B~Reed.
\newblock {\em Perfect Graphs}, volume~44.
\newblock Wiley, 2001.

\bibitem[Viz64]{vizing}
Vadim Vizing.
\newblock On an estimate of the chromatic class of a p-graph.
\newblock {\em Diskret. Analiz}, 3(7):25--30, 1964.

\bibitem[Viz65]{vizing2}
Vadim Vizing.
\newblock Critical graphs with given chromatic class.
\newblock {\em Diskret. Analiz}, 5(9):9--17, 1965.

\end{thebibliography}

\end{document}